\documentclass[11pt]{amsart}  
\usepackage{amssymb, latexsym, amsmath, amsfonts}
\usepackage{float}

\restylefloat{figure}
\usepackage{graphicx}
\usepackage{xcolor}

\newtheorem{thm}{Theorem}[section]
\newtheorem{cor}[thm]{Corollary}
\newtheorem{lem}[thm]{Lemma}
\newtheorem{prop}[thm]{Proposition}
\theoremstyle{definition}    
\newtheorem{defn}[thm]{Definition}
\theoremstyle{remark}
\newtheorem{rem}[thm]{Remark}  
\newtheorem{ex}[thm]{Example}
\numberwithin{equation}{section}

\begin{document}
\author{G. P. Balakumar}
\address{G.P. Balakumar: 
Indian Institute of Technology, Palakkad, India}
\email{gpbalakumar@gmail.com}
\pagestyle{plain}

\begin{abstract}
It is a classical fact that domains of convergence of power series of several complex variables are characterized as logarithmically convex complete Reinhardt domains; let $D \subsetneq \mathbb{C}^N$ be such a domain. We show that a necessary as well as sufficient condition for a power series $g$ to have $D$ as its domain of convergence is that it 
admits a certain decomposition into elementary power series; specifically, $g$ can be expressed as a sum of a sequence of power series $g_n$ with the property that each of the logarithmic images $G_n$  of their domains of convergence are  half-spaces, all containing the logarithmic image $G$ of $D$ and such that the largest open subset of $\mathbb{C}^N$ on which all the $g_n$'s and $g$ converge absolutely is $D$. In short, every power series admits a decomposition into elementary power series. The proof of this leads to a new way of arriving at a constructive proof of the aforementioned classical fact. This proof inturn leads to another decomposition result in which the $G_n$'s are now wedges formed by intersections of pairs of {\it supporting} half-spaces of $G$. Along the way, we also show that in each fiber of the restriction of the absolute map to the boundary of the domain of convergence of $g$, 
there exists a singular point of $g$. 
\end{abstract} 

\title{Structure theorems for Power Series in Several Complex Variables.}
\maketitle

\section{Introduction.}

\noindent Decomposing functions in some general class of functions as a sum of special functions in that class, is a broad theme in mathematics that needs no further motivation. A purpose of this article, among others, is to attain decomposition results for functions defined by power series of several complex variables, valid on their domains of convergence.
It is very well-known that the domain of convergence of every power series $g(z)$ of $N$ complex variables say, $z=(z_1, \ldots ,z_N)$
is a logarithmically convex complete Reinhardt domain $D \subset \mathbb{C}^N$ -- we shall assume unless otherwise mentioned, that $g$ is a convergent power series i.e., that $D$ is non-empty (and for the most part that $D \neq \mathbb{C}^N$). In the case 
$N=1$, every such domain is just a disc, possibly of infinite radius in which case the domain is just $\mathbb{C}$ and unless this happens, the domains of convergence (in this case $N=1$) are all biholomorphically, infact linearly, equivalent. In contrast to this case which is too simple in this context, as soon as the dimension $N$ gets bigger than one, these domains are in general never biholomorphically equivalent. Nevertheless they share a key common feature namely, as already mentioned, their images under the standard logarithmic map 
\[
\lambda(z)= (\log \vert z_1 \vert, \ldots, \log \vert z_N \vert)
\]
(applied to only those $z$ none of whose coordinates is zero), are convex domains in 
$\mathbb{R}^N$. Recall that every convex domain can be realized as the intersection of its supporting half-spaces (unless we are in the trivial case that the convex domain is $\mathbb{R}^N$ which is the only exception) and in this sense, half-spaces are both the simplest and special instances of convex domains, after the whole space $\mathbb{R}^N$ itself. It immediately follows therefore that every logarithmically convex complete Reinhardt (proper) domain in $\mathbb{C}^N$ can be expressed as the intersection of elementary Reinhardt domains i.e., domains obtained as inverse images of half-spaces in $\mathbb{R}^N$, under the map $\lambda$. Corresponding to these expressions of general domains in terms of simple ones, we shall show first that every (convergent) power series can be realized as a (countable) sum of elementary power series, where by an elementary power series, we mean a power series with the property that the logarithmic image of the domain of its convergence is a half-space in $\mathbb{R}^N$ or $\mathbb{R}^N$ itself. This result is a triviality in dimension one, as every power series in one variable is already an elementary power series. Thus, we shall focus only on the case $N>1$ without further mention, though our results hold true in the case $N=1$ as well (for trivial reasons!). The ideas in the proof of the aforementioned result leads to a constructive proof that every logarithmically convex complete Reinhardt domain is the domain of convergence of some power series, which inturn leads to a variation of the aforementioned decomposition result, all explained below. Converse of these decomposition results also holds and put together, these may be viewed as decomposition/structure theorems for power series of several complex variables. The proofs require working out some refined versions of results in convex analysis, which may be of 
some independent interest as well. \\

\noindent Postponing the recollection of notions from convex analysis and some non-standard terminology to the next section, we now proceed towards
stating some of the main results of this article. \textcolor{blue}{However, we preface (as well as postface) the statements with a couple of remarks, so as to help the reader discern the idea of the results even without
the full knowledge of the precise meanings of a few of the terms involved. To elucidate the first theorem about the decomposition into elementary power series as already indicated, let us begin with some simple ways of writing down examples of elementary power series. To this end, fix any 
$J \in \mathbb{N}_0^N \setminus \{0\}$ (where $\mathbb{N}_0 = \mathbb{N} \cup \{0\}$) and consider the sequence $\{J^k \; : \; k \in \mathbb{N}\}$
defined by $J^k = k J $ for all $k \in \mathbb{N}$ i.e., the sequence consisting of all positive integral multiples of the fixed 
$N$-tuple $J$. Consider then, any power series of the form
\[
f(z) = \sum_{k\in \mathbb{N}} c_{k} z^{J^k} = \sum_{k\in \mathbb{N}} c_{k} z^{kJ}.
\]
Strictly speaking, we have not written $f$ in its usual standard form -- monomials $z^J$ for $J \not \in \langle J \rangle := \{ kJ : k \in \mathbb{N}\}$ do not make an explicit appearance in the above, meaning that their coefficients are all zero (and are therefore not written, as is also usual). However, this does not mean that the $c_k$'s above are all non-zero; it only means that we are considering such a power series $f$ whose non-zero coefficients are among the $c_k$'s as above.
The logarithmic image $G_f$ of the domain of convergence $D_f$ of $f$ may then be written explicitly (by an analogue of the Cauchy -- Hadamard formula for power series in several complex variables; see propositions \ref{Shabatprop} and \ref{logimagecharac} of the preliminaries section) as: 
\[
G_f = \{ s \in \mathbb{R}^N \; : \; \limsup_{k \to \infty} \left( \langle J^k/\vert J^k \vert, s \rangle +  
\log \vert c_k \vert^{1/\vert J^k \vert} \right)\; < 0 \}.
\]
Note that the members of the sequence here, namely $\{J^k = kJ\}$ as mentioned above, all have the same ratio in common namely, $J^k / \vert J^k \vert = J /\vert J \vert$ where
$\vert \cdot \vert$ denotes the $l^1$-norm. Let $\alpha$ denote this common ratio i.e., $\alpha= J/\vert J \vert$ and $c_f$ the constant
\[
c_f= \frac{1}{\vert J \vert} \left( \limsup_{k \to \infty} \log \vert c_k \vert^{1/\vert J^k \vert} \right).
\]
With these notations, we may rewrite $G_f$ also as,
\[
G_f = \{ s \in \mathbb{R}^N \; : \; \langle \alpha, s \rangle +  c_f < 0 \} 
\]
which, unless $G_f$ is empty or the whole space $\mathbb{R}^N$ (depending upon whether $c_f=\infty$ or $c_f=-\infty$), is clearly a half-space with gradient $\alpha \in PS_N$ where $PS_N$ denotes the `probability simplex' (or the non-negative face of the standard simplex $S_N$, the unit ball with respect to the $l^1$-norm) given by 
\[
PS_N = \{ s \in \mathbb{R}^N \; : \; s_1 + \ldots + s_N =1, \; \text{and } s_j \geq 0 \text{ for all } j \}.
\]
Let $\pi$ denote the `radial' projection onto $S_N$ i.e., the map $\pi: \mathbb{R}^N \setminus \{0\} \to S_N$ given by 
\[
\pi(z) = \frac{z}{\vert z \vert_{l^1}} = \frac{1}{\vert z \vert_{l^1}} (z_1, \ldots, z_N),
\]
and note that $\pi(\mathbb{N}_0^N \setminus \{0\})$ is contained in $PS_N$; indeed, $\pi(\mathbb{N}_0^N \setminus \{0\})=PS\mathbb{Q}^N$, the set of all points on $PS_N$ with rational coordinates. Stated differently, the inverse image under $\pi$ of $PS\mathbb{Q}^N$ is $\mathbb{N}_0^N \setminus \{0\}$ which is the set that indexes the summation in any power series in $N$-variables when written in its usual standard form.
Therefore, we observe that $\mathbb{N}_0^N \setminus \{0\}$ may be spilt up 
as the {\it disjoint} union of the countably many sequences $\{J^k_l \; : \; k \in \mathbb{N}\} $ indexed by $l \in \mathbb{N}$ with the property that 
$\pi(J^k_l)=\alpha_l$ where $\{\alpha_l \; : \; l \in \mathbb{N}\}$ is some enumeration of the rational points in $PS_N$. This observation leads us to the following decomposition of any given power series $g(z) = \sum c_J z^J$ into elementary power series
\begin{equation} \label{naiveeqn}
g(z)= \text{constant term} + \sum_{l \in \mathbb{N}} g_l(z)
\end{equation}
where $g_l$ is the elementary power series given by 
\[
g_l(z) = \sum_{k\in \mathbb{N}} c_{J^k_l} z^{J^k_l}.
\]
That this is indeed an elementary power series (i.e., the logarithmic image of the domain of its convergence is a half space in $\mathbb{R}^N$ or 
all of $\mathbb{R}^N$ itself)
is a consequence of the analogue of the Cauchy -- Hadamard formula already mentioned and will be detailed later in the sequel. Note that for each fixed $l \in \mathbb{N}$, the $J^k_l$'s (for all $k \in \mathbb{N}$) are all integer multiples of $\alpha_l$; thereby, we may 
view the summation in the decomposition (\ref{naiveeqn}) as being indexed by $\alpha_l$'s (rather than by $l \in \mathbb{N}$) i.e.,
 by $PS\mathbb{Q}^N$, the set of all rational points on $PS_N$.
However, this decomposition is a bit too naive, atleast if one wants to use this idea to write down a concrete power series whose domain
of convergence matches with a prescribed logarithmically convex complete Reinhardt domain in $\mathbb{C}^N$, as done in theorem \ref{mainthm} below. Moreover, we may indicate here itself, the naivety of this
decomposition briefly with a simple example as follows. Consider an elementary power series in two complex variables $z,w$ of the form 
$f(z,w) = \sum z^{j_n}w^{k_n}$ where $(j_n,k_n) \in \mathbb{N}^2$ have the property that $(j_n,k_n)/(j_n + k_n) \to \alpha$ for some non-rational point $\alpha \in PS_2$; for concreteness say, $\alpha=(\sqrt{2} -1, 1)/\sqrt{2} \not \in PS \mathbb{Q}^2$, the set of all points on $PS_2$ with 
rational coordinates. Then, it is not difficult to see that the domain of convergence of $f$ is $\lambda^{-1}(G_f)$ where $G_f$ is the 
half-space in $\mathbb{R}^2$ bounded by the line described by the equation $(\sqrt{2}-1)x + y=0$ whose gradient is $\alpha$. While we may very well write a decomposition of the form given by equation (\ref{naiveeqn}) above for $f$, it must be noted that the $g_l$'s therein always have the feature that the logarithmic images of their domains of convergence are half-spaces bounded by hyperplanes with {\it rational} gradient or the whole space $\mathbb{R}^N$ itself. Applied to the two-variable power series $f$ just-mentioned, it can be verified with some contemplation that all of the corresponding $g_l$'s are polynomials and thereby that their domains of convergence are all $\mathbb{C}^2$ itself (thereby their intersection 
is $\mathbb{C}^2$ as well) whereas, the domain of convergence of $f$ is a proper subdomain of $\mathbb{C}^2$. It must be noted also in this example that, the supporting hyperplane at each point of $\partial G_f$ is the line $\partial G_f$ itself which has irrational slope; and, no hyperplane (which in the setting of this example is a line in $\mathbb{R}^2$)  with rational slope, has
even the property that a half-space bounded by it contains $G_f$, let alone support $\partial G_f$ (since every such hyperplane intersects $\partial G_f$). In short, the domains of convergence of the
$g_l$'s that occur in (\ref{naiveeqn}) above may well not have a tangible relationship to that of $g$ and so, the decomposition (\ref{naiveeqn}) is more of a formal one. To deal inclusively with examples such as these, we are led to start with a countable 
dense subset of the effective domain $C_h$ of the support function $h=h_g$ of the convex domain $G =G_g \subset \mathbb{R}^N$ given by the logarithmic image of the 
domain of convergence $D=D_g$ of the given power series $g$. Basic facts about domains of convergence of power series imply that the image of the convex cone $C_h$ under the map $\pi$ namely $\pi(C_h)$, must be contained in the non-negative face of the simplex (namely $PS_N$ in the aforementioned notation); we denote $\pi(C_h)$ by $PS_h$ and call it the {\it normalized effective domain} of $h$. To spell out one last time briefly, the necessity to introduce all this, note that the normalized effective domain of the support function of $G_f$ -- in the aforementioned example $f$ of a power series in two complex variables -- is the singleton $\{\alpha\}$ which is a non-rational point in $PS_2$, thereby it is not possible to obtain a decomposition of the form (\ref{naiveeqn}) such that the intersection of the domains of convergence of the elementary power series given by the summands $g_l$ on the right of (\ref{naiveeqn}) equals or even (atleast) contains that of $f$. One final terminology, before we proceed to lay down the full statement our first main theorem: we say that an elementary power series is proper, if the logarithmic image of the domain of its convergence is a proper half-space (i.e., a domain in $\mathbb{R}^N$ bounded by a hyperplane, so that in particular it is not all of $\mathbb{R}^N$); we shall assume $D \neq \mathbb{C}^N$ in the theorems below for convenience (for otherwise, the results are in a sense trivial)}. 

\begin{thm} \label{eldecompthm}
Every power series can be decomposed as a sum of elementary power series; infact, these elementary power series can be chosen to be proper as soon as the domain of convergence of the given power series is not all of $\mathbb{C}^N$ and to have several
properties as detailed next. Let $D \subsetneq \mathbb{C}^N$ be any logarithmically convex complete Reinhardt domain. Let $h=h_G$ be the support function of the convex domain 
in $\mathbb{R}^N$ given by its logarithmic image $G=\lambda(D)$; let $PS_h$ denote the normalized effective domain of the support function $h$. For any given countable dense subset $\mathcal{C}$ of $PS_h$, we have the following.
For every power series 
$g(z) = \sum c_J z^J$ with $D$ as its domain of convergence, there exists a sequence of elementary power series 
$\{g_n(z) \; : \; n \in \mathbb{N}_0 \}$ such that:
\begin{itemize}
\item[(i)] the $g_n$'s are mutually monomial-wise completely different power series, all of which are sub-series of the given series $g$
and $g_0$ is the constant term of $g$, 
\item[(ii)] if we write $g_n$ in standard form as $g_n(z) = \sum\limits_{k \in \mathbb{N}} c_{J^n_k} z^{J^n_k} $, then 
$\pi(J^n_k) = J^n_k/\vert J^n_k \vert$ converges as $k \to \infty$, to the point $\alpha^n$ belonging to the prescribed set $\mathcal{C}$. Furthermore, the coefficients of each $g_n$ satisfy
\[
\limsup_{k \to \infty} \log \vert c_{J^n_k} \vert^{1/\vert J^n_k \vert } = -c(\alpha^n).
\]
where $c(\alpha)$ is the function on $PS_N$ defined in terms of the coefficients of $g$ by
\[
c(\alpha) =  - \sup \big\{\limsup_{n \to \infty} \big(\frac{\log \vert c_{J^n} \vert}{\vert J^n \vert} \big) \; : \; \{J^n\} \in S_{\alpha} \big\},
\]
\item[(iii)] the largest open subset of $\mathbb{C}^N$ where all the $g_n$'s \textcolor{red}{and $g$} converge absolutely is $D$, on which we also have the $g_n$'s rendering a decomposition of $g$ (into elementary power series) in the sense that: $\sum g_n(z) = g(z)$ holds for all $z \in D$. 
\end{itemize} 
\end{thm}

\begin{rem}
It will be shown in the course of proving this theorem that its condition (ii) can be interpreted geometrically. Namely, it is tantamount to the condition that the logarithmic image $G_n=\lambda(D_n)$ of the domain of convergence $D_n$ of each $g_n$ is a half-space, {\it with normal vector belonging to the prescribed set} $\mathcal{C} \subset PS_h$ i.e., $G_n$ is of the form
\[
G_n = \{ s \in \mathbb{R}^N \; : \; \langle \alpha^n, s \rangle - c(\alpha^n) <0 \}.
\]
It will be shown that the convex closure of the function $c$ above is the support function of $G$. This means in particular that each $G_n$ contains the supporting half-space to $G$ with normal vector $\alpha^n$ and the theorem can be paraphrased by saying that every power series $g$ can be 
decomposed as a sum of mutually monomial-wise completely different, transversal (in the sense of definitions \ref{1stpowdefn}, \ref{2ndpowdefn} below) elementary power series $g_n$'s all of whose domains of convergence contain that of $g$. So, this means  that condition (ii) actually implies part of the assertion of (iii) in the above theorem namely, that the largest open subset of $\mathbb{C}^N$ on which all the $g_n$'s converge absolutely contains $D$. We have nevertheless mentioned this explicitly for some convenience at this stage; it must also be noted that part-(iii) does not say that the intersection of the $D_n$'s equals $D$ but rather that it contains $D$. Also owing to this formulation of condition (iii), the converse of the theorem holds trivially. For a more non-trivial converse, see proposition
\ref{partialconverse} and the remark \ref{remonc-condn} below.
\end{rem}

\noindent It will also be shown that the intersection of the half-spaces $\{ s \in \mathbb{R}^N \; : \; \langle \alpha^n, s \rangle - c(\alpha) <0 \}$ for $\alpha$ running through $PS_N$, equals $G$. Now, the foregoing theorem leaves to be desired that the intersection of the $D_n$'s reduce precisely to 
$D$ (it will be seen that this happens if the function $c$ in (ii) above, 
is convex\footnote{It seems somewhat unlikely that the half-spaces which are the logarithmic images of the domains of convergence of the summands $g_n$'s can always be ensured to be supporting half-spaces for $G$, particularly if the condition that the summands be sub-series of the given series has also to be met, among the other conditions of theorem \ref{eldecompthm}; this is based on matters discussed later in the article, in particular on the fact that the convex closure of a function is almost never equal to the function itself. However, this is only a conjecture of this author and could be wrong as well!}). This is ensured in the following theorem, which however does not ensure that the summands in the decomposition are sub-series of the given series. To summarize the essence of this theorem in a single sentence, we first make the following definition.  

\begin{defn}
We say that a power series is {\it simple}, if the logarithmic image (in $\mathbb{R}^N$) of the domain of its convergence is expressible as the intersection of finitely many half-spaces; in such a case, we say that the domain of convergence is simple.
\end{defn}

\begin{thm} \label{structurethm}
Every power series can be decomposed as a sum of countably many simple power series with several properties; to be more precise and informative,
let $D \subsetneq \mathbb{C}^N$ be any logarithmically convex complete Reinhardt domain. Let $h=h_G$ be the support function of the convex domain 
in $\mathbb{R}^N$ given by its logarithmic image $G=\lambda(D)$; let $PS_h$ denote the normalized effective domain of the support function $h$. For any given countable dense subset $\mathcal{C}=\{\alpha^n\}$ of distinct elements of $PS_h$, we have the following.
Every power series 
$g(z) = \sum c_J z^J$ with $D$ as its domain of convergence can be decomposed
as a sum of countably many simple power series $\sigma_n(z)$ ($n \in \mathbb{N}_0$)
such that:
\begin{itemize}
\item[(i)] the domain of convergence of each $\sigma_n$ is a simple logarithmically convex complete Reinhardt domain $D_n$ whose
logarithmic image $\lambda(D_n) \subset \mathbb{R}^N$ is actually a wedge i.e., the intersection of just two half-spaces bounded by hyperplanes whose normal vectors 
are the distinct elements $\alpha^n,\alpha^{n-1}$ belonging 
to the prescribed countable dense subset $\mathcal{C}$, 
\item[(ii)] the pair of half-spaces whose intersection gives the wedge $\lambda(D_n)$ are both supporting half-spaces of the convex domain $G$; thereby, the intersection of the $D_n$'s is precisely $D$,
\item[(iii)] the largest open subset of $\mathbb{C}^N$ where all the $\sigma_n$'s converge absolutely is $D$, on which 
the $\sigma_n$'s render a decomposition of $\sigma$ (into simple power series) in the sense 
that $\sum \sigma_n(z) = g(z)$ for all $z \in D$; the summation here runs from $n=0$ to $\infty$, and $\sigma_0$ is the constant term of $g$.
\end{itemize} 
\end{thm}

\noindent In the subsection that intermediates between the proofs of the foregoing pair of theorems, we shall give as an illustration of the ideas in the proof of the first theorem, a constructive proof of the classical fact that every logarithmically convex complete Reinhardt domain $D$ is the domain of convergence of some power series; indeed, as theorem \ref{eldecompthm} asserts that every power series is decomposable into elementary power series, it is only natural to construct a power series obtained by summing up elementary power series, the logarithmic images of whose domains of convergence are supporting half-spaces for $G=\lambda(D)$. We hasten to mention that constructive proofs of this fact is not entirely new. Infact, they can be traced back to the more than a century old seminal work 
\cite{Hartogs} of Hartogs in German (the translation of which into English is as yet unavailable) and it was possibly rediscovered by Russian authors in \cite{AM} who however addressed the case only when $D$ is bounded; a better exposition can be found in the lecture notes by H. Boas \cite{Bo} who has also developed the proof to tackle the case when $D$ is unbounded. Our proof here is somewhat different from these and will not split into cases depending on whether $D$ is bounded or not. Moreover, the ideas in this proof will inturn be employed in the proof of the above theorem \ref{structurethm}. For now, let us only lay down the statement here. 

\begin{thm}\label{mainthm}
Let $D$ be any given logarithmically convex complete Reinhardt domain 
in $\mathbb{C}^N$ with its logarithmic image denoted by $G$.
Let $h$ denote the support function of this convex domain $G \subset \mathbb{R}^N$, 
with $PS_h$ denoting its normalized effective domain. Let 
$\{ \alpha^n : n \in \mathbb{N} \}$ be any given countable dense subset of $PS_h$. Pick any 
doubly-indexed sequence
$\{ J^{nk} : n,k \in \mathbb{N}\}$ of distinct points from 
 $\mathbb{N}_0^N \setminus \{0\}$ with the property that
 for each $n \in \mathbb{N}$, $\pi (J^{nk}) \to  \alpha^n$ 
as $k \to \infty$ (such a sequence always exists). Then, for each $n \in \mathbb{N}$, the elementary power series with positive coefficients 
given by
\[
f_n(z) = \sum_{k \in \mathbb{N}} c_{nk}z^{J^{nk}} 
\]
{\bf \textcolor{red}{where $c_{nk} =e^{-\vert J^{nk} \vert h(\alpha^n)}$}} has the logarithmic image of its domain of convergence equalling a supporting half-space of the given $G$; in particular, they all converge absolutely on the given $D$.
Furthermore, their sum
\begin{equation} \label{splpowseries1}
f(z)= \sum f_n(z) = \sum\limits_{n,k \in \mathbb{N}} c_{nk}  z^{J^{nk}},
\end{equation}
which is indeed a power series 
by the distinctness of $J^{nk}$'s (no `like-terms' have to be combined), converges absolutely and uniformly on $D$ and $D$ is the domain of convergence of the power series given by $f$ above. 
\end{thm}

\noindent It is then straightforward to recover as a corollary, the well-known theorem that logarithmically convex complete Reinhardt domains $D$ are domains of holomorphy. Indeed, this will be done by showing first that the power series for $D$ as provided by the above theorem (whose coefficients are positive) is completely singular at the absolute boundary $\vert \partial D \vert$ of $D$ i.e., at all boundary points of $D$ with non-negative coordinates; we may also think of the absolute boundary as the image of the boundary $\partial D$ under the absolute map $z \to (\vert z_1 \vert, \ldots, \vert z_N \vert)$. To obtain a holomorphic function on the given $D$ that is singular at a boundary point not from the absolute boundary part of $\vert \partial D \vert$, one then only has to pre-compose $f$ with suitable rotations of the coordinates; this is needed, for the series of the above theorem need not be singular at boundary points off $\vert \partial D \vert$, as can be seen in the case of $D=\Delta^n$ the standard polydisc, in which case the series provided by the above theorem is the standard $N$-variable geometric series (much as the theorem attempts to provide the simplest power series whose domain of convergence is the given $D$). It must also be noted here that while the power series given by the above theorem is competely singular at all points of $\vert \partial D \vert$, it can be shown that for every power series $g$ with domain of convergence
$D$, the subset $S_g$ of $\partial D$ consisting of all the singular points of $g$ is not only always nonempty but infact has the property that its absolute image exhausts $\vert \partial D \vert$ i.e., 
$\vert S_g \vert=  \vert \partial D \vert$, as will be proven in proposition \ref{singsetarbpow} below. Stated differently, in each fiber of the restriction of the absolute map  to $ \partial D$,
there exists a singular point of $g$. \\

\noindent To conclude here, we would like to highlight a result in the final section, namely theorem \ref{charbysuppfn}, not just because it is new but because it is among the key tools in the proofs of the above main theorems; to obtain this result, we develop another technical but fundamental result, one in convex analysis namely theorem \ref{cvxgeomfound}, which we presume would be of some independent interest as well.\\

{\it Acknowledgements}: I thank Prof. Grigorchuk of the Texas A{\&}M University for asking some doubts about certain erroneous points in an earlier expository article on power series of mine \cite{B}, which first provided the impetus to set them right and then develop it further into this article. Prof. Boas of the same University is also thanked herewith, for making his notes on Several Complex Variables publicly and freely available which has been (among others as well) a source of inspiration much as the book \cite{S} by Shabat.

\section{Preliminaries}
In this section, we recall well-known facts relevant to this article from two subjects namely, convex analysis and several complex variables. We shall at times, deduce some simple consequences of these known results for later use in the article and introduce a couple of new terms (with precise definitions, needless to say) as well.
\subsection{Convex Analysis}
\noindent Convex functions in this article are allowed to take values in the extended reals, though for the most part we shall only be concerned with those that do not assume the value $-\infty$. While convex functions $f:X \to (-\infty, \infty]$ are assumed to be defined on convex subsets $X \subset \mathbb{R}^N$, we may always extend to all of $\mathbb{R^N}$ preserving convexity by setting its values equal to $+\infty$ at all points  where it is not apriori given i.e., on $\mathbb{R}^N \setminus X$; we shall tacitly assume this to have been done and continue to denote this extension by $f$.  If we were to carve out some of the simplest subclasses of convex functions on $\mathbb{R}^N$, perhaps the first subclass would be the class of affine-linear ones which are however trivial in the sense of being both convex and concave (their negatives are also convex); but these are the only such 
functions. To pass to the next degree of simplicity, we may consider the class of functions 
whose epigraphs are convex cones. It may first be noted here that the functions whose epigraphs are cones (not necessarily convex but invariant under scaling by positive numbers) are precisely those functions $f$ on $\mathbb{R}^N$ which are {\it positively homogeneous} meaning that $f$ satisfies $f(tx) = t f(x)$ for all $t \geq 0$ and $x \in \mathbb{R}^N$. Skipping the proof of this fact, we however detail  one of its consequences namely,
the following lemma which despite being an equally easy fact, is of much importance for us.

\begin{lem}\label{poslem}
Let $f$ be a positively homogeneous function and $A(x) = \langle a,x\rangle +b$ any affine function
majorized by $f$. Then $b$ must be non-positive and the corresponding homogeneous-linear polynomial (also linear functional) 
namely, $L_a(x) = \langle a,x \rangle$ is also majorized by $f$.
\end{lem}
\begin{proof}
By our definition of positive homogeneity $f(0)=0$ and so the majorization of $A$ by $f$ at the origin 
forces the non-positivity of the constant $b$. More importantly fix any $x$ and write out the
majorization of $A$ by $f$ at the point $tx$, combined with the positive homogeneity of $f$ 
as: 
\[
 \langle a, tx \rangle - b \leq f(tx)  = t f(x).
\]
Dividing throughout by $t$, we get $L_a(x)-b/t  \leq f(x)$ which holds for all positive $t$ thereby allowing us to
let $t \to \infty$, to render $L_a(x) \leq f(x)$. As $x$ was arbitrary, this shows the desired majorization of $L_a$ by $f$ or equivalently that the (conic) epigraph of $f$ is contained in the half-space
bounded by $\ker(L_a)$, into which $a$ points. 
\end{proof}

\noindent Instances of positively homogeneous functions of fundamental importance in convex analysis and for this article, are those of support functions, defined as follows.
\begin{defn} \label{suppfn}
Let $C \subset \mathbb{R}^N$ be a closed convex set. The {\it support function} $h=h_C: \mathbb{R}^N \to (-\infty, +\infty]$ of $C$ is defined by
\[
h(u) = \sup \{ \langle x, u \rangle \; : \; x \in C \}.
\]
The set of all $u \in \mathbb{R}^N$, for which $h(u)$ is finite is called the {\it effective} domain of $h$ and we call its subset consisting unit vectors
thereof, as the {\it normalized effective} domain of $h$.
\end{defn}

\noindent The geometric meaning of the support function is: for a unit vector $u$ with $h(u)$ finite, the number $h(u)$ is the signed distance of the supporting hyperplane to $C$ with normal vector $u$, from the origin; the distance is negative if and only if $u$ points into the open half-space containing the origin. From the definition, it is straight-forward to check that $h_C(\cdot) = \langle z, \cdot \rangle$ is a linear functional iff $C$ is a singleton (namely, $\{z\}$). More importantly, $h$ is {\it positively homogeneous}:  $h(\lambda u) = \lambda h(u)$ for all $\lambda \geq 0$  and is sub-additive:
\[
h(u +v) \leq h(u) + h(v).
\]
It follows in particular that $h$ is a convex function. If $x \in \mathbb{R}^N \setminus C$, a version of the Hahn-Banach separation theorem yields the existence of a vector $u_0$ with
$\langle x, u_0 \rangle > h(u_0)$. The consequence of such separation theorems, of interest to us is the following.

\begin{thm} \label{cvxsetrep}
Let $S_N$ denote the unit sphere of $\mathbb{R}^N$ with respect to any norm on it. Let $C \subset \mathbb{R}^N$ be a closed convex set with support function $h$ whose normalized effective domain we denote by $S_h$. Then, $C$ can be expressed as the intersection of its supporting half-spaces as:
\[
C= \bigcap_{u \in S_h} \{ s \in \mathbb{R}^N \; : \; \langle u, s \rangle - h(u) \leq 0\}
\]
The same holds if $C$ is an open convex set as well, provided we replace the strict inequalities by non-strict ones.
\end{thm}

\noindent While the positively homogeneous functions are linear along each ray, little can be said in general of their joint regularity.  The positively homogeneous functions of first interest for us are the support functions of convex sets, which are ofcourse convex and hence enjoy more regularity 
than arbitrary positively homogeneous functions. It must however be first observed that the domain
of definition of any support function or for that matter any positively homogeneous function can be 
well be taken to be the unit sphere $S_N$ (with respect to any chosen norm) for their values are
then completely determined on all of $\mathbb{R}^N$ by their homogeneity property. Conversely, given any function on the unit sphere, it can be extended (uniquely) to all of $\mathbb{R}^N$ as a positively homogeneous function. We shall however have occasions where we take the (restricted)
unit sphere as the domain and other occasions, where we take the whole space as the domain of such functions.
The regularity of support functions is ofcourse derived from their convexity and we therefore now turn 
to the continuity of convex functions in general, which is always guaranteed on the interior of their 
domains. However, we hasten to recall that their boundary behaviour can arbitrarily wild: 

\begin{ex} \label{Excvxbdybeh}
Consider 
for any function $f: \mathbb{R}^2 \to \mathbb{R}$ which takes the value $+\infty$ on the 
complement of the closed unit disc $\overline{\mathbb{D}}$ centered at the origin in $\mathbb{R}^2$, vanishes identically on the open unit disc $\mathbb{D}$ whereas on its boundary $\partial \mathbb{D}$,
it is set equal to a arbitrarily chosen positive real-valued function $f_b$ on $\partial \mathbb{D}$.
Note that whatever the choice of $f_b$, the resulting function is indeed convex!
\end{ex}

\noindent The aforementioned continuity of convex functions is formulated more generally and precisely in the following theorem, whose proof can be found in many standard treatises on convex analysis such as \cite{Rocka} or \cite{H}. 

\begin{thm} \label{contofcvxfns}
If $f$ is a convex function on $\mathbb{R}^N$, then 
\[
X = \{ x \in \mathbb{R}^N \; : \; f(x) < \infty \}
\]
is a convex set and $f$ is continuous in the relative interior of $X$ i.e, in the interior of $X$, the affine-hull of $X$.
\end{thm}

\begin{rem}
It is not always possible to redefine $f$ at boundary points of $X$ in ${\rm ah}(X)$, so as to have $f$ become
continuous with values in $(-\infty, +\infty]$.
\end{rem}

\noindent This problem is redressed by taking the lower-semicontinuous regularization (see \cite{H})
\begin{prop}
Let $f: \mathbb{R}^N \to (-\infty, \infty]$ be any function. Define for all $x \in \mathbb{R}^N$:
\[
{\rm cl } f(x)= \liminf\limits_{y \to x} f(y).
\]
Then this function ${\rm cl} f$ is lower semi-continuous and is called the lower semi-continuous regularization/lower semi-continuous hull/the lower semicontinuous closure, of $f$. \\
If $f$ is convex, then ${\rm cl } f$ is also convex and ${\rm cl } f(x) \leq f(x)$ for all $x$, with equality if $x$ lies in the interior of $X=\{ x \in \mathbb{R}^N : f(x)<\infty\}$ in 
${\rm ah}(X)$ or interior in $\mathbb{R}^N \setminus X$.
\end{prop} 
\noindent Note the trivia that ${\rm cl } f(x) \leq f(x)$ for all $x \in \mathbb{R}^N$ and infact, we may define ${\rm cl}f$ as the largest lower semicontinuous minorant of $f$. The role of lower semi-continuity here is explained as follows. While the epigraph of a function $f$ is convex iff its epigraph is convex, the epigraph is closed iff $f$ is lower semi-continuous. This will be important in the brief discussion about the Legendre transform to be recalled below. While the epigraph of a convex function $f$ on $\mathbb{R}^N$ gives a convex subset of $\mathbb{R}^{N+1}$, we may also extract a convex function on $\mathbb{R}^N$ out of any convex subset of $\mathbb{R}^{N+1}$ (see \cite{R}).

\begin{thm}\label{lowbdconstrcvxfn}
Let $F$ be any convex subset of $\mathbb{R}^{N+1}$. Then,
\[
f(x) := \inf\{ y \; :\; (x, y) \in F\}
\]
is a convex function on $\mathbb{R}^N$ (where the infimum of the empty subset of real numbers is defined to be $+\infty$).
\end{thm}

\noindent This actually leads to use any function to generate a convex function.
\begin{defn}
Let $g: \mathbb{R}^N \to (-\infty,\infty]$ be any function (not necessarily convex). Its {\it convex hull} is
denoted ${\rm conv} g$ defined as the function as the function obtained by applying the foregoing theorem \ref{lowbdconstrcvxfn} to the convex hull of the epigraph of $g$ i.e. by taking 
$F= {\rm ch}({\rm epi}(g))$ therein.
Equivalently, ${\rm conv} g: \mathbb{R}^N \to (-\infty,\infty]$ is the greatest convex function majorized by $g$.
\end{defn}

Note that the epigraph of ${\rm conv}g$ is not automatically closed as exemplified by example
(\ref{Excvxbdybeh}); the epigraph of ${\rm cl}({\rm conv} g)$ is the closure of the convex hull of the epigraph of $g$.
\begin{lem}
If $g: \mathbb{R}^N \to (-\infty,\infty]$ is a positively homogeneous function then so are ${\rm conv} g$, 
${\rm cl} g$ and ${\rm cl}({\rm conv} g)$.
\end{lem}

\noindent While the simplest examples of convex functions are the affine functions, it turns out that 
we may atleast pointwise
approximate any closed convex function by such functions, as expressed by the following fundamental theorem of convex analysis.

\begin{thm} \label{cvxanalfundthm}
Every closed convex function is the pointwise supremum of the collection of all affine functions $h$
such that $h \leq f$.
\end{thm} 

\noindent Now a corollary of theorem \ref{cvxanalfundthm}, of great importance for us is:
\begin{cor} \label{cvxmain}
For every function $f: \mathbb{R}^N \to (-\infty, \infty]$, its convex closure ${\rm cl}({\rm conv} f)$ equals the pointwise
supremum of the collection of all affine functions on $\mathbb{R}^N $ majorized by $f$.  If $f$ is positively homogeneous, then we may instead of {\it all} affine functions, reduce the afore-mentioned
collection to the subfamily of homogeneous linear functions (i.e., affine functions vanishing at the origin) majorized by $f$ (and ${\rm cl}({\rm conv} f)$ is also positively homogeneous).
\end{cor}

\begin{proof}
As ${\rm cl}({\rm conv} f)$ is the greatest closed convex function majorized by $f$, the affine
functions $h$ such that $h \leq {\rm cl}({\rm conv} f)$ are the same as those such that 
$h \leq f$ due to the basic lemma \ref{cvxmaj} about majorization of convex functions below. The claim about restricting the collection over which the supremum is taken for the case of positively homogeneous functions then follows by lemma \ref{poslem}.
\end{proof}

\begin{lem} \label{cvxmaj}
Let $g$ be any $\overline{\mathbb{R}}$-valued function and $f$ a convex function majorized by $g$ i.e., $f(x) \leq g(x)$ for all $x$. Then, $f$ must be majorized by ${\rm conv} g$, the convex hull of  $g$, as well. If $f$ is (convex and) closed, then $f$ must be majorized by the convex closure of $g$: $f \leq {\rm cl}({\rm conv} g)$.
\end{lem}

\begin{proof}
The hypothesis means the containment of the epigraphs of the pair of functions in the statement: 
${\rm epi} f \supset {\rm epi} g$. The convexity of $f$ which is tantamount to the convexity of the
set ${\rm epi} f$ then implies that this epigraph must contain the convex hull of the epigraph of $g$,
which when interpreted back in terms of functions only means $f \leq {\rm conv} g$, finishing the proof of the first statement. Reasoning likewise
with the operation of taking the lower semicontinuous hull/closure using the basic topological fact that 
a closed set $C$ containing another set $B$ must also contain its closure $\overline{B}$, yields the second statement of the lemma.
\end{proof}

\noindent Theorem \ref{cvxanalfundthm} leads to the fact that the Legendre transform (also called the Legendre--Fenchel transform) is an involution on the space of closed convex functions; recall that the Legendre transform can actually be defined for any function $f$ and it encodes the family of all affine functions majorized by $f$. More precisely,
\begin{defn}
Let $f: \mathbb{R}^N \to \mathbb{R}$ any function. The Legendre transform ($=$ Fenchel -- Legendre transform), also called the convex conjugate, of $f$,  is defined by
\[
f^*(y) = \sup_{x} \{ \langle x, y \rangle - f(x) \}.
\]
\end{defn}

\noindent We note in passing here that the support function of a convex set $C$ may also be equivalently defined as the Legendre transform of its indicator function $I_C$, which in turn is defined as follows.

\begin{defn}
Let $E \subset \mathbb{R}^N$. The indicator function $I_E$ is the function whose value at points of $E$ is set equal to $0$ and equal to $+\infty$ at all points outside $E$. Such a function is convex precisely when $E$ is convex.
\end{defn}

\noindent Finally here, we turn to a notion of convexity on spheres, which we will come up once explicitly in our discussions later and is good to have in the background for certain other matters in the sequel as well. To introduce this
quickly, consider the unit sphere centered at the origin in $\mathbb{R}^N$ denoted $S^{N-1}$ -- the sphere here can be taken with respect to any norm; if it is the $l^1$-norm then this coincides with what we shall deal with numerous times later and denote by the lower subscript namely, the simplex $S_{N-1}$. Let $S \subset S^{N-1}$. The cone $C_S$ 
consisting of rays through the origin and passing through points $p$, as $p$ varies through $S$, is referred to as the cone spanned by $S$.
If $C_S$ happens to be a convex cone in $\mathbb{R}^N$, then 
we say that $S$ is a spherically convex subset of $S^{N-1}$. It is easily seen that a spherically convex subset is either $S^N$ itself or contained in
a closed hemisphere thereof. A great circle on $S^{N-1}$ is the intersection of a two dimensional linear subspace of 
$\mathbb{R}^N$ with $S^{N-1}$. Points $x,y \in S^{N-1}$ are said to be antipodal if $y = -x$. Note that for a pair of points which are not antipodal, there is a unique great circle passing through them. A spherical-geometric way of checking convexity of a subset $A$ of the sphere, is to check 
for every pair of points $x,y \in A$ with $y \neq \pm x$ that, the set $A$ contains the smaller arc of the great circle on $S^{N-1}$ connecting $x$ and $y$; and for every pair of antipodal points in $A$ that, at least one-half of one of the great circles through them, lies in $A$.  
We refer the reader to 
\cite{W} for more about spherical convexity.

\subsection{Several Complex Variables}

\begin{defn} \label{1stpowdefn}
Let $f(z) = \sum c_J z^J$ be a formal power series. Denote by $B$ the set of all points of $\mathbb{C}^N$ at which the series $S$ converges absolutely; its interior $D=D_f=B^0$ is termed the domain of convergence of the power series $S$. We say that {\it the monomial $z^J$ with coefficient $c_J$ occurs} in the power series $f$, iff $c_J \neq 0$. If $g(z) = \sum d_J z^J$ is another power series, we say that $f,g$ are {\it mutually monomial-wise completely different} iff none of the monomials occurring in one of them occurs in the other, though it can well happen that some monomials occur in neither of them.
\end{defn}

\noindent In this article, we shall only be concerned with convergent power series i.e., power series $f$ for which $D_f \neq \emptyset$ and shall not qualify such $f$ by the adjective `convergent' any longer. Infact, for the most part we shall only be dealing with the case $D_f \neq \mathbb{C}^N$. Evidently, the domain of convergence $D_f$ is  the largest open subset
 of $\mathbb{C}^N$ on which the series $f$ converges absolutely. It is very well-known that such domains are completely determined by their absolute images i.e., images under the natural map $\tau: \mathbb{C}^N \to \mathbb{R}_+^N$ (where $\mathbb{R}_+$ is the set of all non-negative reals)
given by 
\[
\tau(z) = (\vert z_1 \vert, \ldots,\vert z_N \vert)
\]
which we call the absolute map. For both this notation and other standard notions here, we refer the reader to \cite{R} and \cite{FG}. The image of a domain $D$ under $\tau$ will be denoted interchangeably by $\vert D \vert$ as well as by $\tau(D)$; when we say that a domain $D$ is completely determined by its absolute image we mean $\tau^{-1}(\tau(D))=D$ in which case $D$ is called Reinhardt. We shall only be dealing with such domains $D$, whose boundary $\partial \tau(D) \subset \mathbb{R}_+^N$ -- which we refer to as the {\it absolute boundary} of $D$ -- must be noted to be subset of the boundary of $D$ itself. Note that $\partial \tau(D) = \tau(\partial D)$ which will again, be interchangeably denoted by 
$\vert \partial D \vert$ or equivalently by $\partial \vert D \vert$. 
Recall further that $D$ is said to be a 
{\it complete Reinhardt} domain if the open polydisc $P_p$ spanned by each point $p$ of the absolute boundary $\vert \partial D \vert$, none of whose coordinates is zero, is contained in $D$; here the open {polydisc spanned by} $p$ is the polydisc centered at the origin with polyradius given by $p$. Note that $D$ is a complete Reinhardt domain iff it is a union of concentric polydiscs, all centered at the origin; indeed, such a domain $D$ may be written as the union of polydiscs spanned by points of $\vert \partial D \vert$ with all its coordinates non-zero. A first fundamental fact about domains of convergence of power series is that they are complete Reinhardt domains
 whose logarithmic images inturn are convex subsets of $\mathbb{R}^N$. As the simplest examples of convex subsets of $\mathbb{R}^N$ apart from
$\mathbb{R}^N$ itself, are half-spaces, which are also special in the sense that an arbitrary convex set can be expressed as the intersection of the half-spaces which contain it, the subclass of logarithmically convex complete Reinhardt domains in $\mathbb{C}^N$ whose logarithmic image is a half-space are singled out and regarded as elementary. 

\begin{defn} \label{2ndpowdefn}
Suppose $D$ is a logarithmically convex complete Reinhardt domain in $\mathbb{C}^N$. If $\lambda(D)$ is a half-space or the whole space $\mathbb{R}^N$, then we say that
 $D$ is an {\it elementary} logarithmically convex complete Reinhardt domain. If $\lambda(D)$ is an intersection of finitely many
  half-spaces, then we say that  $D$ is a {\it simple} logarithmically convex complete Reinhardt domain. A power series $f = \sum c_J z^J$ whose domain of convergence is an elementary logarithmically convex Reinhardt domain will be called an {\it elementary power series}. Two elementary power series are said to be {\it transversal} if the bounding hyperplanes in $\mathbb{R}^N$ of the loagirhmic images of their domains of convergence, have non-parallel normal vectors (and hence intersect transversally). 
\end{defn}

\begin{rem}
If $c_J=0$ for all but finitely many $J$, then ofcourse $f$ is a polynomial and these are trivial instances of (convergent) elementary power series much as their domains of convergence is $\mathbb{C}^N$, an elementary Reinhardt domain -- we regard $\mathbb{R}^N$ also as a half-space for convenience for otherwise, we would have to exclude $\mathbb{C}^N$ itself from being regarded as an elementary Reinhardt domain! 
\end{rem}

\begin{defn}
A polydisc $U=U(z^0,r)$ is termed a polydisc of convergence of $\sum c_J z^J$ if $U \subset B$ but in any polydisc $U(z^0, R)$ where each $R_j \geq r_j$ for $j=1, 2, \ldots, N$ with at least one of the inequalities being strict, there are points in $U(z^0, R)$ where the series diverges. Every such polyradii $(r_1,r_2,\ldots,r_n)$ of $U(z^0,r)$ is called an {\it associated polyradii}  i.e., the radii of each polydisc of convergence are called {\it associated radii of convergence}.
\end{defn}
\noindent The associated radii (also called conjugate polyradii as for instance in Shabat's text \cite{S}) satisfy an analogue of the Cauchy -- Hadamard formula expressing the radius of convergence of a power series in terms of its coefficients, as in the following proposition.
\begin{prop} \label{Shabatprop}
The associated radii of convergence of the power series $\sum\limits_{k=1}^{\infty}\sum_{\vert J \vert=k} c_J z^J$ satisfy the relation
\begin{equation}\label{R}
\limsup_{\vert J \vert \to \infty} \sqrt[\vert J \vert]{\vert c_J r^J \vert}=1
\end{equation}
i.e., $\limsup\limits_{k\to \infty} \max\{|c_J r^J|^{1/k}: |J|=k\} = 1$.
\end{prop}

\noindent The proof of this proposition as in Shabat's book \cite{S} can be adapted to obtain a multidimensional extension of a basic fact  about singular points of power series with non-negative coefficients, in a single complex variable known as Pringsheim's theorem which will be needed much later in the sequel.
The extension of Pringsheim's theorem to power series of several variables which was recorded with proof in \cite{SS} as noted by Boas in his notes \cite{Bo} who also gives a proof; nevertheless, we record a proof here as we feel it is somewhat simpler than these earlier treatments, by direct reductions to the one-dimensional case. But first let us clarify what singular points of power series mean, by laying down a precise definition.

\begin{defn}
Let $f$ be a power series whose domain of convergence $D=D_f$ is neither empty nor the whole space $\mathbb{C}^N$ so that $\partial D \neq \emptyset$. We say that $p \in \partial D$ is a {\it regular point} for $f$, if there exists an open polydisc $U$ centered at $p$ and a holomorphic function $g$ on 
$U$ such that $g_{\vert_{U \cap D}} = f_{\vert_{U \cap D}}$; $g$ is referred to as a direct analytic continuation of $f$ across $\partial D$ near $p$. If a boundary point $p \in \partial D$ is not a regular point for $f$, we say that $p$ is a {\it completely singular point} (or just 
a {\it singular point}, for short) of $f$. The set of all singular points of $f$ will be denoted by $S_f$.
\end{defn}


\noindent It follows immediately from the definition that the set of all regular points of $f$ forms an open subset of $\partial D$, thereby that $S_f$ is a closed subset of $\partial D$. When the dimension $N=1$ (and with the condition $D_f \neq \mathbb{C}^N$ as already stated), the set of singular points $S_f$ is always a compact set. However, as soon as  $N>1$, $S_f$ need not be compact, much as $\partial D$ need not; an example is furnished by the Taylor series about the origin in $\mathbb{C}^2$ of the rational function $1/(1-zw)$ in the pair of complex variables $z,w$, which is an elementary power series. But perhaps it must first be noted that $S_f$ is indeed non-empty when $N>1$ as well, even though the argument when $N=1$ uses compactness of
 $\partial D$. We defer the proof of this to the next section, where we shall show something much stronger namely, there is a singular point of $f$ in the fiber over every point of $\vert \partial D \vert$ under the absolute map $\tau$. For now, let us get back to the aforementioned extension of Pringsheim's theorem.
 
\begin{lem}\label{multiPringsheim}
Let $g$ be a power series whose coefficients are non-negative. Then $g$ is completely singular at every point of 
$\partial D \cap \mathbb{R}_+^N$ where $D$ is the domain of convergence of $g$. In particular, the holomorphic function defined by the power series provided by theorem \ref{mainthm}, is completely singular at every absolute boundary point of the domain of its convergence. 
\end{lem}

\begin{proof}
Let $r=(r_1, \ldots,r_N)$ be an absolute boundary point of $D$ and consider the restriction of $g$ to the one-dimensional slice of $D$ obtained by intersecting $D$ with the complex line $L_r$ spanned by $r$ namely,
\[
g_r(\zeta) = \sum c_J r^J \zeta^J.
\]
This is strictly speaking not a power series in $\zeta$ unless we combine the like terms together to write it as
\[
\sum_{k=1}^{\infty} \big( \sum_{\vert J \vert=k} c_J r^J \big) \zeta^k.
\]
We already know this converges for all $\zeta$ in the standard unit disc $\Delta \subset \mathbb{C}$, because $g$ converges on $D \cap L_r$ in particular. Now, if there exists $\zeta_0 \in \mathbb{C} \setminus \overline{\Delta}$ at which $g_r(\zeta)$ converges, then first of all, it must be convergent at all points of the dilated disc $\vert \zeta_0 \vert \cdot \Delta$. This implies that the coefficients of $g_r$ namely, 
$\sum_{\vert J \vert=k} c_J r^J$ satisfy the Cauchy estimate
\[
\big \vert \sum_{\vert J \vert=k} c_J r^J \big\vert \leq \frac{M}{\vert \zeta_0 \vert^k}
\]
for some positive constant $M$. As the $c_J$'s are non-negative, this means in particular that 
\[
c_J r^J \leq \frac{M}{\vert \zeta_0 \vert^k}
\]
for all $J \in \mathbb{N}_0^N$ with $\vert J \vert =k$. However, as $k$ is allowed to take all values of $\mathbb{N}$, this amounts to just saying that 
\[
c_J = \vert c_J \vert \leq \frac{M}{\vert \zeta_0 r \vert^{\vert J \vert}}
\]
holds for all $J \in \mathbb{N}_0^N$. But this estimate only means that our original power series $g(z)$ converges in the polydisc $P$ centered at the origin with polyradius $\zeta_0 r=(\zeta_0 r_1, \ldots, \zeta_0 r_N)$. But then $P$ contains $r$ as an interior point, leading to the contradiction that $r$ was a boundary point of the domain of convergence of $g$. Hence, the disc of convergence of $g_r(\zeta)$ is precisely the unit disc $\Delta$ and by the (one variable) Pringsheim's theorem, $1 \in \partial \Delta$ is a singular point or $g_r$ and consequently, $r$ must be a point where $g$ is completely singular. As $r \in \vert \partial  D \vert$ was arbitrary, this completes the proof.
\end{proof}

\noindent As we know well, power series are the building blocks of holomorphic functions as they are always representable {\it locally} by some power series (centered at various points; though we restrict to power series centered at the origin in this article for simplicity, without loss of generality needless to say). A natural basic question in multlidimensional complex analysis is about conditions which lead to such local representations becoming a global one; this is answered by the following proposition which does not suppose that the domain is logarithmically convex.

\begin{prop}\label{singlepowser}
Let $D$ be any complete Reinhardt domain in $\mathbb{C}^n$. Then, the Taylor series expansion of every holomorphic function $f$ about the origin in $D$, converges to $f$ absolutely and uniformly on compact subsets of $D$; in particular, $f$ is represented by a single power series on $D$. 
\end{prop}

\noindent Next, rewriting equation (\ref{R}) by taking logarithms and denoting $\log r_j$ by $s_j$, the relation (\ref{R}) reads:
\begin{equation}\label{Rpsi}
\limsup \limits_{\vert J \vert \to \infty} \Big( \frac{j_1 s_1 + j_2 s_2 + \ldots + j_N s_N }{ j_1 + j_2 + \ldots +j_N} +
 \log \vert c_J \vert^{1/\vert J \vert} \Big) = 0.
\end{equation}
\noindent Denoting the function on the left hand side by $\psi(s_1, \ldots,s_N)$, we note that the above equation expresses $\psi$ as the limsup of a family, infact a sequence, of affine functions. Thus, $\psi$ must be convex. The domain of convergence $D$ of our given power series corresponds to the domain $G=\{ s \; : \; \psi(s) < 0 \}$. Let us rewrite this more precisely and record it for now: $D = \{ z \in \mathbb{C}^N \; : \; 
\varphi(z) <0\}$ where $\varphi$ is given in terms of the coefficients of our power series $\sum c_J z^J$ by
\begin{equation} \label{deffnforbdy}
 \varphi(z_1, \ldots, z_N) = \limsup_{\vert J \vert \to \infty} \sqrt[\vert J \vert]{\vert c_J z^J \vert} \;\;- 1.\\
\end{equation}
\noindent We would like to note down here for one last time, that the limsup in (\ref{deffnforbdy}) -- which will occur numerous times in the sequel -- is an abbreviation for 
\[ 
\limsup_{k \to \infty} ( \max\{ \vert c_J z^J \vert^{1/k} : \vert J \vert =k \}.
\]
Similar understanding applies to all occurrences of limsup of countable subsets of real numbers indexed by $\mathbb{N}^N$, prior to (\ref{deffnforbdy}) as well; for instance, (\ref{Rpsi}) means that the domain $G$, the logarithmic image of the domain of convergence, 
comprises of those points $s \in \mathbb{R}^N$  which satisfy
\begin{equation} \label{Logiimagedefn}
 \limsup_{k \to \infty} 
\bigg( \max \big\{ \langle \frac{J}{\vert J \vert}, s \rangle + 
\frac{\log \vert c_J \vert}{\vert J \vert} \; : \; \vert J \vert=k \big\} \bigg) <0 .
\end{equation}
A bit more elaborately in the familiar meaning of limsup in terms of subsequential limits, observe that the above condition for a point $s$ to belong to $G$ can be equivalently described as in the following 

\begin{prop} \label{logimagecharac}
The logarithimic image $G_g$ of the domain of convergence of a power series $g(z) = \sum c_J z^J$ consists 
precisely of those points $s \in \mathbb{R}^N$ for which every `subsequential limit' of the form
\[
\lim_{n \to \infty} \Big( \langle \frac{J^n}{\vert J^n \vert}, s \rangle + 
\frac{\log \vert c_{J^n} \vert}{\vert J^n \vert} \Big)
\]
where $\{J^n\}$ is a sequence of distinct elements of $\mathbb{N}^N$ such that the above limit exists,
is negative. This is abbreviated by writing 
\begin{equation}\label{logimageshortfrm}
G_g= \big\{ s \in \mathbb{R}^N \; : \; \limsup_{\vert J \vert \to \infty}
\Big( \langle \frac{J}{\vert J \vert}, s \rangle + 
\frac{\log \vert c_J \vert}{\vert J \vert}\Big) <0 \big\}.
\end{equation}
\end{prop}

\begin{proof}
The proof being obvious, shall be limited to a few sentences for the record. When (\ref{Logiimagedefn}) is given to hold, then every subsequential limit as in the statement of this proposition, is clearly forced to be negative. Conversely, suppose for {\it every} sequence $\{J^n\} \subset \mathbb{N}^N$,
 the associated subsequential limit as in the assertion, is negative. Then, first pick 
 any subsequence $\{k_n\}$ of the natural numbers (which in particular means that $k_n$'s are distinct) and let $J^n$ be an index for which the maximum in
 \[
\max \big\{ \langle \frac{J}{\vert J \vert}, s \rangle + \frac{\log \vert c_J \vert}{\vert J \vert} \; : \; \vert J \vert=k_n \big\} 
 \]
is attained. Note then that $\vert J^n \vert =k_n \to \infty$ as $ n \to \infty$ and that the $J^n$'s are indeed distinct elements of $\mathbb{N}^N$. 
Finally just by supposition, we have
\[
\lim_{n \to \infty} \Big( \langle \frac{J^n}{\vert J^n \vert}, s \rangle + \frac{\log \vert c_{J^n} \vert}{\vert J^n \vert}   \Big)  < 0,
\]
which by definition of $J^n$ means that (\ref{Logiimagedefn}) holds, completing the proof.
\end{proof}

\noindent Henceforth, we shall interpret limsup (and in an analogous fashion, liminf) of all sequences of real numbers that
are indexed by $\mathbb{N}^N$ to be encountered in the sequel, as in the above proposition. As can be surmised from equation (\ref{Rpsi}) and the above proposition, the supporting half-spaces have gradients with non-negative rational coordinates. Indeed, this can be established rigorously 
which we skip for brevity and we only record the result in the foregoing notations.

\begin{prop}
Let $q \in \partial G_g$. Then, every supporting hyperplane $H^q$ at $q$ to $\partial G$, is defined by
an affine linear functional of the form $A_q(s) := u_1s_1+ \ldots + u_N s_N + d^q$
where $d^q$ is the signed-distance, measured in the $l^\infty$-norm, of $H^q$ from the origin; and, 
$u=(u_1, \ldots,u_N)$ is a unit vector in the $l^1$-norm, with all its coordinates non-negative. Thus, the supporting half-spaces for $G$
are of the form $ \{s \in \mathbb{R}^N \; : \; A_q(s)<0\}$.
\end{prop}

\noindent We shall denote the $l^1$-norm on $\mathbb{R}^N$ simply by $\vert \cdot \vert$; its unit sphere in this norm is the standard 
simplex $S_N$ and its intersection with the non-negative orthant of $\mathbb{R}^N$ popularly called the probability simplex, is given by
\[
PS_N = \{ x \in \mathbb{R}^N \; :\;  x_1  + \ldots +  x_N  = 1, \text{ and }  x_j  \geq 0 \text{ for all } j\}
\]
which may be noted to be the convex hull of the standard basis of $\mathbb{R}^N$. The foregoing proposition implies that the normalized effective domain of the support function which we earlier denoted $S_h$, is contained in $PS_N$ and henceforth for emphasis on this feature, we shall denote
$S_h$ by $PS_h$.\\

\noindent Rational points on $PS_N$ will come up recurrently in the sequel as they have already; we introduce some notations for them. For all $K \in \mathbb{N}^N \setminus \{0\}$, let 
\[
\pi(K)=(\frac{k_1}{k_1+ \ldots +k_N}, \ldots , \frac{k_N}{k_1+ \ldots +k_N})
\]
where $\pi: \mathbb{R}^N \to S_N$ denotes the projection onto the unit sphere $S_N$ given by $\pi(z)=z/\vert z \vert$. The set of all such points
$\pi(K)$ forms precisely the set $PS{\mathbb{Q}^N}$ of points on $PS_N$ with rational coordinates (which needless to say, is a countable dense subset of $PS_N$). Other $l^p$-norms 
(for $1\leq p \leq \infty$) if needed, will be denoted by  $\vert \cdot \vert_{l^p}$.\\ 

A notation that will be recurrently used in the sequel throughout is the following. For each $\alpha \in PS_N$, we denote by $S_\alpha$, the set of all 
sequences $\{J^n\}$ of distinct elements in $\mathbb{N}_0^N\setminus \{0\}$ with $\pi(J^n) \to \alpha$ as $n \to \infty$; here and as always throughout, 
$\mathbb{N}_0 = \mathbb{N} \cup \{0\}$. 

\section{Proofs of the main theorems}
\noindent By a fundamental theorem of convex geometry namely theorem \ref{cvxsetrep}, we may write $G$ as:
\begin{equation} \label{cvxdombasic}
G= \bigcap_{\alpha \in PS_h} \{s \in \mathbb{R}^N \; : \; A_\alpha(s) <0 \}
\end{equation}
where $A_\alpha(s)= \langle \alpha, s\rangle - h(\alpha)$, $h$ is the support function of the convex domain $G$ with $PS_h$ denoting its {\it normalized effective domain} defined by
\[
PS_h = \{ \alpha \in S_N \; : \; h(\alpha) \text{ is finite } \}.
\]
However from the point of view of our goal, (\ref{cvxdombasic}) requires fine-tuning
as this is an expression of $G$ as an intersection of supporting half-spaces parametrized by 
$\alpha \in PS_h$ which in general is uncountable, whereas we will need to 
re-express (\ref{cvxdombasic}) in countable-terms. This is rendered in general, by the following  result valid for all closed convex sets
which does not seem to be available in the literature and therefore provide a proof, as it maybe of some independent interest as well.
\begin{thm} \label{cvxgeomfound}
Let $C \subset \mathbb{R}^N$ be any closed convex set with support function $h$ whose normalized effective domain we denote and define by
\[
S_h = \{ \alpha \in S_N \; : \; h(\alpha) \text{ is finite } \}= (E_h \setminus \{0\}) \cap S_N
\] 
where $E_h$ is the effective domain of $h$.
Let $\mathcal{C}$ be any countable dense subset of $S_h$.
Then
\begin{equation} \label{countblexp}
C = \bigcap_{\alpha \in \mathcal{C}} \{ s \in \mathbb{R}^N \; : \; A_\alpha(s) \leq 0 \}
\end{equation}
where $A_\alpha(s) = \langle \alpha, s \rangle - h(\alpha)$.
\end{thm}

\begin{proof}
We shall do this in two steps, first reduce the parametrizing set for the half-spaces 
involved in the fundamental representation of convex domains, from $S_h$ to its relative
interior ${\rm relint}(S_h)$ and thereafter in the next step, further to 
$\mathcal{C}^i := \mathcal{C} \cap {\rm relint}(S_h) \cap \mathcal{C}$. So, for the first step we begin
with the just-mentioned fundamental representation of $C$ as the intersection of half-spaces:
\[
C = \bigcap_{\alpha \in S_h} \{ s \in \mathbb{R}^N \; : \; A_\alpha(s) \leq 0 \}.
\] 
We wish to replace $S_h$ by ${\rm relint}(S_h)$ which we may trivially do if the
effective boundary $\partial_e S_h := S_h \setminus {\rm relint}(S_h)$ is empty. If not, 
pick $\beta \in \partial_e S_h$ and observe that $\beta$ is an extreme point of a line segment 
$l$ contained in $S_h$, owing to the fact that $S_h$ is the intersection of the convex cone $E_h \setminus \{0\}$  with the simplex $S_N$. We may actually take $l$ such that $l^0=l \setminus \{\beta\}$ is contained within ${\rm relint}(S_h)$ and we shall henceforth assume so. Restricting the convex function $h$ to this line segment $l$ gives a one-variable
convex function on an interval $I \subset \mathbb{R}$ parametrizing $l$, with one of its 
end-points corresponding to $\beta$; note that the this restriction of $h$ to $l$ (or $l^0$) gives a {\it bounded} function, since $\beta$ is a point of the effective boundary $\partial_e S_h$. By basic estimates about the boundary behaviour of one-variable
convex functions on intervals in $\mathbb{R}$, we have with $l^0$ denoting $l \setminus \{\beta\}$:
\[
h(\beta) \geq \lim_{l^0 \ni \alpha \to \beta} h(\alpha)
\]
where the limit on the right is guaranteed to exist by the fact that (bounded) one-variable convex functions on bounded intervals have one-sided limits at the end-points. Temporarily, let
\[
C' = \bigcap_{\alpha \in {\rm relint}(S_h)} \{ s \in \mathbb{R}^N \; : \; A_\alpha(s) \leq 0 \}.
\]
We are to show $C' \subset C$ to finish our aforementioned first step. To this end, pick any
$s \in C'$. Then, as $l^0 \subset {\rm relint}(S_h)$, we have $\langle \alpha, s \rangle \leq  h(\alpha)$ for {\it all} $\alpha \in l^0$ and so, taking 
the limit as $\alpha \to \beta$, we get 
\[
\langle \beta, s \rangle \leq \lim_{l^0 \ni \alpha \to \beta} h(\alpha) \leq h(\beta).
\]
This means that 
$s \in \overline{H}_\beta := \{ s \in \mathbb{R}^N \; : \; 
\langle \beta,s \rangle - h(\beta) \leq 0 \}$. As $\beta \in \partial_e S_h$ was arbitrary and as we already have $s \in \overline{H}_\beta$
if $\beta \in {\rm relint}(S_h)$ just by definition of $S$, we have altogether $s \in H_\alpha$ for all $s \in S_h$. Thus $s \in \mathcal{C}$. As $s \in C'$ was arbitrary, we have established 
$C' = C$, our 
aforementioned first step.\\

\noindent Now, to finish the proof of the theorem, observe that it suffices to show that if
$s \in \mathbb{R}^N$ satisfies the countably many constraints $A_\alpha(s) \leq 0$ for 
$\alpha$ running through $\mathcal{C}^i : =\mathcal{C} \cap {\rm relint}(S_h)$, then it automatically
satisfies $A_\alpha(s) \leq 0$ for all $\alpha \in {\rm relint}(S_h)$ because $\mathcal{C}^i : =\mathcal{C} \cap {\rm relint}(S_h)$ is also dense in $S_h$. To do so, let $\beta$ now be an arbitrary point of 
${\rm relint}(S_h)$. Apparently, this essentially requires the continuity of the support function of $h$ which however, is not guaranteed on all of $S_h$  (counterexamples exist to falsify such a claim in general). But then again by theorem \ref{contofcvxfns}, $h$ is indeed always continuous on the relative interior ${\rm relint}(S_h)$ of $S_h$ and this suffices to complete the proof as follows. Pick any sequence $\alpha^n \in \mathcal{C}^i$ converging to $\beta$, so that taking limits as $n \to \infty$ in
\[
A_n(s) = \langle \alpha_n, s \rangle - h(\alpha_n) \leq 0
\]
leads, by the continuity of $h$ at $\beta \in {\rm relint}(S_h)$ to $A_\beta(s) \leq 0$. 
As $\beta \in S_h$
was arbitrary, this means that $s$ lies in each of the closed supporting half-spaces for $G$ and thereby that $s$ lies in the closure of $G$. As $s$ was an arbitrary point of the right-hand-side of (\ref{countblexp}), this finishes the proof.
\end{proof}

\noindent So, the above theorem when applied to our particular case $C = \overline{G}$, where $G$ is the logarithmic image of a log-convex complete Reinhardt domain in $\mathbb{C}^N$, implies that we may write:
\begin{equation} \label{ourcountblexp}
C = \bigcap_{\alpha \in \mathcal{C}} \{ s \in \mathbb{R}^N \; : \; A_\alpha(s) \leq 0 \}
\end{equation}
for any countable dense subset $\mathcal{C}$ of $PS_h$, the normalized effective domain
of the support function of $G$ (which coincides with the support function of its 
closure $\overline{G}$); note that $PS_h$ is contained within the positive face of the simplex $S_N$. Having such an expression
for $\overline{G}$  in terms of its support function,
we are now motivated to first derive out of (\ref{Logiimagedefn}) -- which is a formula for the defining function -- an expression for the support function
of $G_g$ for a {\it given} power series $g$ in terms of the coefficients of $g$. We begin with the following rephrased version of (\ref{logimageshortfrm}).
\begin{prop}
For each $\alpha \in PS_N$, let $S_{\alpha}$ be the set of all sequences $\{J^n\}$ in $\mathbb{N}^N$ with $J^n/\vert J^n \vert \to \alpha$ and with $\vert J^n \vert \to \infty$ as $n \to \infty$. Define the constants $d_\alpha$ by 
\[
d_\alpha = \sup \big\{\limsup_{n \to \infty} \big(\frac{\log \vert c_{J^n} \vert}{\vert J^n \vert} \big) \; : \; \{J^n\} \in S_{\alpha} \big\}.
\]
Then:
\begin{equation} \label{rephrased}
G_g = \bigcap_{\alpha \in PS_N} \big\{s \in \mathbb{R}^N \; : \; \langle \alpha, s \rangle + d_\alpha <0 \big\},
\end{equation}
\end{prop}
\begin{proof}
We must show the equality of the the already proven expression of $G_g$ in (\ref{logimageshortfrm}) with the intersection in the above statement which we shall briefly denote by $G_g'$ until this proof is complete. To this end, pick $s_0 \in G_g$ i.e., a point $s_0$ satisfying
\begin{equation}\label{sysofineq}
\langle \frac{J}{\vert J \vert}, s_0 \rangle + \frac{\log \vert c_J \vert}{\vert J \vert} < 0
\end{equation}
for all but finitely many $J \in \mathbb{N}^N$. Pick any $\alpha \in PS_N$ and any sequence 
$\{J^n\} \subset \mathbb{N}^N$ for which $J^n/\vert J^n  \vert$ converges to $\alpha$. By (\ref{sysofineq}), we have
\[
\frac{\log \vert c_{J^n} \vert}{\vert J^n \vert}< -\langle\alpha,s_0 \rangle
\]
for all but finitely many $n \in \mathbb{N}$. As this holds for every sequence $J^n$ with
$J^n/\vert J^n  \vert$ converging to $\alpha$, we get 
\[
\sup \big\{\limsup_{n \to \infty} \big( \frac{\log \vert c_{J^n} \vert}{\vert J^n \vert} \big) \; :
 \; \{J^n\} \in S_{\alpha} \big\} \; \; < \; -\langle \alpha, s_0 \rangle
\]
i.e., $s_0$ lies in the open half-space 
$H_\alpha = \{ s \in \mathbb{R}^N \; :\; \langle \alpha,s\rangle + d_\alpha <0 \}$. As $\alpha$
was just an arbitrary point of $PS_N$, we conclude therefore that 
\[
s_0 \in  \bigcap_{\alpha \in PS_N} H_\alpha \; = G_g'.
\]
As $s_0 \in G_g$ was arbitrary, this finishes one direction of the proof, namely $G_g \subset G_g'$. Conversely next, let $s^0 \in G_g'$ and pick any sequence 
$\{J^n\} \subset \mathbb{N}^N$ with the property that
\[
\lim_{n \to \infty} \Big(\langle \frac{J^n}{\vert J^n \vert}, s \rangle + \frac{\log \vert c_{J^n} \vert}{\vert J^n \vert} \Big)  
\]
exists, equal to $L$ say. We need to show $L<0$ to complete the proof of the desired equality $G_g'=G_g$ of this proposition. To this end, pick any limit point $\beta$ of the sequence
 $J^n/\vert J^n \vert$. So,
$\beta = \lim_{k \to \infty}  J^{n_k}/\vert J^{n_k} \vert$  for some subsequence $\{J^{n_k}\; : \; k \in \mathbb{N}\}$ of $\{J^n\}$ and
\[
\langle \beta,s \rangle + \lim_{k \to \infty} \big( 
\frac{\log \vert c_{J^{n_k}} \vert}{\vert J^{n_k} \vert} \big) = L. 
\]
Note that the limit displayed explicitly in the above equation, must exist just by the existence of the limit that defines the number $L$ (and basic limit laws). More importantly, next note that as $G_g'$ is by definition the intersection of the half-spaces $H_\alpha$ for $\alpha$ varying in $PS_N$, 
$s^0 \in G_g'$ means in particular that we have $s^0 \in H_\beta$, which inturn means that $L$ is negative. Thus, {\it every} subsequential limit of the countable subset of reals given by
\[
\Big\{ \langle \frac{J}{\vert J \vert}, s \rangle + 
\frac{\log \vert c_J \vert}{\vert J \vert} : \; \; J \in \mathbb{N}^N \Big\},
\]
is (strictly) negative; so $s^0 \in G_g$. As $s^0$ was an arbitrarily fixed point of $G_g'$, this completes the proof.
\end{proof}
\noindent We remark in passing, that the set $S_\alpha$ involved in the definition of the half-spaces $H_\alpha$ intersecting to 
yield $G_g$ as in the above proposition can be verified to be uncountable by a 
version of the Cantor's diagonal argument. We now move onto drawing out implications of the foregoing proposition. Note that we have by basic convex analysis that 
\[
G_g = \bigcap_{\alpha \in PS_N} \big\{ s \in \mathbb{R}^N \; : \; \langle \alpha, s \rangle  - h(\alpha) <0 \big\}.
\]
Comparing the foregoing pair of representations of $G_g$, using the fundamental fact that for any convex domain, there can be at most one supporting hyperplane 
with a given gradient and combining it with uniqueness of support functions of convex domains, 
we {\it are motivated} to claim the following expression for the support function in terms of the coefficients of the given power series:
\begin{equation*}
h(\alpha) =  -\limsup\limits_{ \{J^n\} \in S_{\alpha}} \frac{\log \vert c_{J^n} \vert}{\vert J^n \vert} := -\sup \big\{\limsup_{n \to \infty} \big(\frac{\log \vert c_{J^n} \vert}{\vert J^n \vert} \big) \; : \; \{J^n\} \in S_{\alpha} \big\}
\end{equation*}
for all $\alpha \in PS_N$. This however may fail to hold. While the aforementioned facts of convex analysis indeed yields a formula for the support function being sought for, in terms of the coefficients of the given power series $g$ the above claim is a bit naive. To point out the issue here briefly, just consider the nice example of the unit ball (with 
respect to the standard $l^2$-norm) $\mathbb{B}^N$ in $\mathbb{R}^N$ or perhaps better still its closure 
$K = \overline{\mathbb{B^N}}$ the compact unit ball, both of which are ofcourse convex. Let $h_K$ denote
as usual the support function (of this simplest instance of a `convex body') and consider the 
representation of $K$ as the intersection of closed half-spaces rendered by it:
\[
K = \bigcap_{\alpha \in \partial \mathbb{B}^N} \{ x \in \mathbb{R}^N \; :\; \langle \alpha,x \rangle
-h_K(\alpha) \leq 0 \}.
\]
Now, observe that replacing the constant terms in the affine functions defining the half-spaces in the above namely $-h_K(\alpha)$, by randomly chosen smaller constants, for $\alpha$ varying through a measure-zero subset of $\partial \mathbb{B}^N$, has no effect on the above intersection and still renders $K$; to bring this point out a bit more clearly, let us do this concretely. Let 
$k(\alpha) = h_K(\alpha)$ for all $\alpha \neq e_1:=(1,0,\ldots,0)$ whereas for $\alpha = e_1$,
let $k(e_1) = h_K(e_1)  + 100$. Then, if we denote by
$K_\alpha = \{ x \in \mathbb{R}^N \; : \; \langle \alpha, x \rangle - k(\alpha) \leq 0\}$ 
for all $\alpha \in \partial \mathbb{B}^N$, the corresponding closed half-spaces, then note that the intersection of the $K_\alpha$'s still render $K$:
\begin{equation} \label{Keqn}
K = \bigcap_{\alpha \in \partial \mathbb{B}^N} \{ x \in \mathbb{R}^N \; :\; \langle \alpha,x \rangle
-k(\alpha) \leq 0 \}.
\end{equation} 
As already indicated, we may likewise alter/increase the values of $h$ at more points than just at one point. But then again, we cannot obviously increase it randomly at `almost all' points; infact we cannot even do so arbitrarily
in any (howsoever small) open subset of $\partial\mathbb{B}^N$, though we shall not rigorously elaborate this here as our purpose was to only point out the main issue with the claim that $h(\alpha) \equiv c(\alpha)$. We are more interested in setting it right, which is done by the following
theorem that gives a characterization of domains of convergence of power series in terms of the
support function of their logarithmic images.
\begin{thm} \label{charbysuppfn}
Let $D$ be any logarithmically convex complete Reinhardt domain in $\mathbb{C}^N$ whose
logarithmic image we denote by $G$. Let $h(\alpha)$ denote the support function of $G$, whose domain we may as well restrict to $PS_N$ i.e., $\alpha$ is restricted to vary only over $PS_N$. Then, the domain of convergence of a
power series $g(z) = \sum c_J z^J$ is $D$ if and only if the support function $h$ of $G$ equals the convex closure of the function $c$ given by
\begin{equation}\label{Main}
c(\alpha) =  - \sup \big\{\limsup_{n \to \infty} \big(\frac{\log \vert c_{J^n} \vert}{\vert J^n \vert} \big) \; : \; \{J^n\} \in S_{\alpha} \big\},
\end{equation}
for all $\alpha \in PS_N$, where (as before) for each $\alpha \in PS_N$, $S_{\alpha}$ denotes the set of all sequences $\{J^n\}$ in $\mathbb{N}^N$ with $J^n/\vert J^n \vert \to \alpha$ and with $\vert J^n \vert \to \infty$ as $n \to \infty$.
\end{thm}

\noindent  The proof of this, is a direct consequence of following proposition which inturn is a straightforward consequence of lemmas in convex analysis recalled in the previous section.
\begin{prop} 
Let $f: \mathbb{R}^N \to (-\infty, \infty]$ be a positively homogeneous function (with $f(0)=0$ as always) and $S_N$ denote the 
unit sphere of $\mathbb{R}^N$ with respect to some given norm. Then
\[
\overline{G} := \bigcap_{\alpha \in S_N} \{ s \in \mathbb{R}^N \; : \; \langle \alpha, s \rangle -  f(\alpha) \leq 0 \}
\]
is a closed convex set whose support function is ${\rm cl}({\rm conv} f)$.
\end{prop}

\begin{proof}
Convexity and closedness of $\overline{G} \subset \mathbb{R}^N$ is an immediate consequence of the fact that both convexity and closedness remain preserved under arbitrary intersections. What needs to 
be established is the claim about the support function which always is a positively homogeneous function and thereby it suffices to restrict attention to its values on $S_N$ for the proof; as $f$ is also positively 
homogeneous as given, it thus suffices to compare the support function $h_{\overline{G}}$ 
with  ${\rm cl}({\rm conv} f)$ on $S_N$. To this end, observe that if $s^0 \not \in \overline{G}$ 
(i.e., $s^0 \in \mathbb{R}^N \setminus \overline{G}$) then 
$\langle \beta, s^0 \rangle \not \leq f(\beta)$ for some $\beta \in S_N$ by the very definition
of $\overline{G}$. This means that the homogeneous-linear functions which are majorized by $f$ are {\it precisely}
those of the form $\langle \cdot, s \rangle$ where $s \in \overline{G}$.
Hence, we have by lemma \ref{cvxmain} that
for all $\alpha \in S_N$ (or for that matter all $\alpha \in \mathbb{R}^N$):
\[
{\rm cl}({\rm conv} f ) (\alpha) = \sup\{ \langle \alpha, s \rangle \; : \; s \in \overline{G} \}.
\]
Noting now that the right hand side is precisely $h_{\overline{G}}(\alpha)$ by definition, finishes the proof.
\end{proof}

\noindent As already noted theorem \ref{charbysuppfn} follows easily from this proposition. What needs greater mention is the answer to the question of whether the operation of convex closure is necessary in theorem \ref{charbysuppfn}. To show the necessity, it suffices to give an example where $c(\alpha)$ actually fails to be convex. To furnish a simple such example, consider the power series in the pair of complex variables $z,w$ given by $f_0(z,w)=\sum c_{jk}z^jw^k$ 
where $c_{jk} = 1$ if $j \neq k$ while $c_{jj}=1+2^j$ for all $j,k \in \mathbb{N}_0$. To understand this power series better note that it is just the sum of the standard double geometric series $f_1(z,w)=\sum z^jw^k$ and the geometric series in the monomial $zw$ namely, 
$f_2(z,w)= \sum 2^j (zw)^j$. The domain of convergence of $f_1$ is the standard unit bidisc $\Delta^2 \subset \mathbb{C}^2$ whose logarithmic image $G_1$ is the third quadrant in $\mathbb{R}^2$; for this series, its $c(\alpha)$ coincides with the support function of $G_1$, which is identically zero on $PS_2$. Next, note that the logarithmic image of the domain of convergence of $f_2$ is a half-space $\mathbb{R}^2$ with boundary line $L$
given by the equation $x+y=\log 2$ (which may be noted to be at a good distance away from $G_1$). Note also that the coefficient of $z^jw^k$ as soon as $j \neq k$ in the series $f_2$ is zero. The domain of convergence $D_0$ of $f_0$ is still the bidisc $\Delta^2$, as is easily seen by examining points $(r_1,r_2) \in \mathbb{R}_+^2$ where this series converges. Finally here, observe that the function $c(\alpha)$ associated to the power series $f_0$, agrees with the support function of $G_1$ (the third quadrant in $\mathbb{R}^2$) except precisely for one point in $PS_2$ which is the point in $PS_2$ that represents the gradient of the line $L$. But the support function of $G_1$ vanishes on $PS_2$ and therefore we conclude that $c(\alpha)$ vanishes on all of $PS_2$ except at its interior point $(1/\sqrt{2}, 1/\sqrt{2})$ where the value of $c$ is strictly positive; indeed, $c(1/\sqrt{2}, 1/\sqrt{2})=\log 2$. Thus the $c(\alpha)$ here is visibly non-convex as can be checked just by definition (or if one wishes to know a way for more such examples which may be complicated, note that the maximum principle for convex functions is violated). Let us also mention in passing here that if it so turns out that for a certain power series $f$ with domain of convergence $D$, its associated $c(\alpha)$ is convex, then ofcourse first of all, it is the support function of the logarithmic image $G=\lambda(D)$ by theorem \ref{charbysuppfn}. Observing that we already know how to write down the defining equation for the boundary $\partial G$ by (\ref{Rpsi}), given the coefficients of $f$, determining $c(\alpha)$ in principle is essentially the question of how to find the support function of a convex domain given its defining function. To answer this question, let $G$ now denote any convex domain in $\mathbb{R}^N$ and recall then that the boundary 
$\partial G$ is smooth almost everywhere (the points where $\partial G$ fails to be smooth is also topologically negligible in the sense being of
the first Baire category), as guaranteed by fundamental theorems of convex geometry which can be found for instance in chapter $5$ of \cite{G}. So it is possible to write 
$\partial G = \{ \rho=0\}$ for some convex function $\rho : \mathbb{R}^N \to \mathbb{R}$ which has a well-defined gradient almost everywhere on 
$\partial G$. It is then straightforward to write down the value of the support function for the normal vector to $\partial G$ which makes sense almost everywhere; specifically, at all smooth points of $\partial G$ as,
\[
h(\triangledown \rho(p)) = \langle p, \triangledown\rho(p) \rangle
\]
which agrees with the Legendre transform of $\rho$ for outer normal vectors at all points of the boundary. If we normalize the normal vectors at all points of $\partial G$, so as to be unit vectors in the $l^1$-norm, the set of all such unit normal vectors form a subset $S$ of $S_N$ called the spherical image of $\partial G$. By virtue of the convexity of $G$ and the study \cite{W} on the spherical images of convex hypersurfaces, made by Wu, 
assures us that the closure $\overline{S}$ of the spherical image of $\partial G$, is spherically convex -- 
(this notion is recalled at the end of the subsection above on preliminaries of convex analysis). The properties of convexity and lower-semicontinuity of the support function can then used to determine the values of the support function throughout $\overline{S}$.\\

\noindent Getting back to theorem \ref{charbysuppfn}, let us only note that it says in essence that just as we have a formula connecting the coefficients of a power series $g$ and the defining function of the logarithmic image $G_g$ of its domain of convergence, we have a similar one linking it to the support function of $G_g$, as well. Now, discussions prior to this theorem, around (\ref{Keqn}), lead to some observations which will useful again in the sequel; we therefore record them in the following remark for later reference.

\begin{rem} \label{GenRepCvxDom}
A convex domain $G$ can be expressed in various ways as an intersection of half-spaces i.e., there are plenty of functions
$f: \mathbb{R}^N \to (-\infty,\infty]$ such that if $H_f(\alpha)$ denotes the half-space $\{ s \in \mathbb{R}^N \; : \; \langle \alpha, s \rangle - f(\alpha) <0 \}$, then the intersection of the $H_f(\alpha)$'s give $G$ as:
\begin{equation} \label{Grepbyf}
G = \bigcap_{\alpha \in S_N} \{ s \in \mathbb{R}^N \; : \; \langle \alpha, s \rangle - f(\alpha) <0 \}.
\end{equation}
However, there is one and only one {\it convex} $f$ such that the above representation is valid, namely the support function of $G$. Moreover, the convex closure of every $f$ such that (\ref{Grepbyf}) is valid (i.e., the intersection of the half-spaces $H_f(\alpha)$ as $\alpha$ runs through $S_N$ renders the given convex domain $G$), is the support function of $G$. \\ 
\end{rem}

\noindent As we shall be expressing a power series as a sum of other power series, a first issue to examine in this regard is the relationship
between the domains of convergence of the summands and the original power series. In this direction, we formulate the following lemmas to begin with.

\begin{lem}\label{domcvgcesuminter}
Let $g_1,g_2, \ldots,g_n$ be (finitely) many power series which are mutually monomial-wise completely different. Then, the domain of convergence of the power series given by their sum $g=g_1 + \ldots + g_n$ equals the intersection of the domain of convergence of the $g_j$'s. For an (infinite) sequence of power series $\{g_n\}$ which are mutually monomial-wise completely different, the domain of convergence of the power series given by their sum $g=\sum_{n=1}^{\infty} g_n$ is contained in the intersection of the domain of convergence of the $g_j$'s.
\end{lem}
\begin{proof}
It is an elementary fact that the domain of convergence of a sum of any finite family of power series contains the intersection of the domains of convergence of the individual members of the family and so, $D \supset D_1 \cap D_2 \cap \ldots \cap D_n$. The hypothesis that the $g_j$'s are mutually monomial-wise completely different means that while summing the $g_j$'s to form $g$, there are then no `like-terms' to be combined. Together with the basic fact that the domain of convergence of every power series is also the domain of its absolute convergence, it therefore follows that the convergence of $g$ at a point $z$ forces the absolute convergence of each of the $g_j$'s and hence $D \subset D_1 \cap D_2 \cap \ldots \cap D_n$ as well, finishing the proof.
\end{proof}

\noindent As the summands involved in our first decomposition theorem are elementary power series, we first give here a way to write down concrete examples of such series; we omit the proof as it is straightforward.  

\begin{lem} \label{half-space-Lem}
Let $\alpha \in PS_N$. Let $\{J^k\} \subset \mathbb{N}^N$ be a sequence with $R_k = \pi(J^k)$ being convergent to $\alpha$. Let $\{c_k\}$ be any 
sequence of complex numbers. Then, the domain of convergence of the associated power series
\[
\sum c_k z^{J^k}
\]
is an elementary Reinhardt domain whose logarithmic image is the half-space given by
\[
\{ s \in \mathbb{R}^N \; : \; \langle \alpha, s \rangle + \limsup_{k \to \infty} \log \vert c_k \vert^{1/\vert J^k \vert} \; <0 \}.
\]
\end{lem}


\begin{lem}\label{samepairlem}
Let $g_1(z) = \sum\limits_{k=1}^{\infty} c_{1,k} z^{I_k}$, $g_2(z) = \sum\limits_{k=1}^{\infty} c_{2,k} z^{J_k}$ be a pair of elementary power series with both the projected sequences of the indices involved $\pi(I_k), \pi(J_k)$ converging as $k \to \infty$, to the same point $\alpha \in PS_N$. Then, the domain of convergence of their sum 
$g(z) = g_1(z) + g_2(z)$ is  the intersection of the domain of convergence of $g_1$ with that of $g_2$, with its logarithmic image given by 
\begin{equation}
G_g = \{ s \in \mathbb{R}^N \; : \; \langle \alpha, s \rangle +d < 0 \}
\end{equation} 
where 
\[
d=\max\{ \limsup_{k \to \infty} \vert c_{1,k} \vert^{1/\vert I_k \vert},\; \limsup_{k \to \infty} \vert c_{2,k} \vert^{1/\vert J_k \vert} \}.
\]
\end{lem}

\begin{proof}
If we write $g(z)$ in standard form as $g(z) = \sum d_J z^J$ then $d_J$ is either zero or $c_{1,k}$ or $c_{2,k}$. Consequently, the supremum of the subsequential limits of the subset of reals
given by 
\[
\{ \vert d_J \vert^{1/\vert J \vert} \; : \; J \in \mathbb{N}_0^N \setminus\{0\}\}
\]
is $d$ as in the statement of the lemma.
Proposition \ref{logimagecharac}, then applies to give the asserted description of $G_g$, finishing the proof.
\end{proof}

\begin{lem}\label{cRealiztn}
For each $\alpha$, we can find a sequence $I^\alpha_k$ of distinct elements of $\mathbb{N}_0^N \setminus\{0\}$ with $\{I^\alpha_k : k \in \mathbb{N}\} \in S_{\alpha}$ (i.e., $\pi(I^\alpha_k) \to \alpha$) and with 
\[
\lim_{k \to \infty} \frac{\log \vert c_{I^\alpha_k} \vert}{\vert I^\alpha_k \vert} = -c(\alpha).
\]
\end{lem}
\begin{proof}
Let us begin with the definition of $c(\alpha)$ namely,
\[
-c(\alpha) = \sup \big\{\limsup_{n \to \infty} \big(\frac{\log \vert c_{J^n} \vert}{\vert J^n \vert} \big) \; :
 \; \{J^n\} \in S_{\alpha} \big\}.
\]
Pick any sequence of elements $\{J^{n,k}\} \in S_{\alpha}$ for $k=1,2,3,\ldots$, such that 
\[
\lim\limits_{k \to \infty} \frac{\log \vert c_{J^{n,k}} \vert}{\vert J^{n,k} \vert} = -c(\alpha).
\]
For each $k \in \mathbb{N}$, pick $n_k \in \mathbb{N}$ such that $J^{n_k,k}$ has the property that
\[
\big\vert \frac{\log \vert c_{J^{n_k,k}} \vert}{\vert J^{n_k,k} \vert} - (-c(\alpha)) \big\vert < \frac{1}{k}.
\]
This means that $I^\alpha_k :=  J^{n_k,k}$ satisfies the limit condition stated in the lemma; observing that 
$J^{n_k,k}/\vert J^{n_k,k} \vert \to \alpha $ we have $I^\alpha_k \in S_\alpha$ as well, finishing the proof.
Note also that after passing to a subsequence if necessary, we may assume -- which we indeed shall assume henceforth -- that the $I^\alpha_k$'s are all distinct (recall that $\vert J^k \vert \to \infty$ as $k \to \infty$ and so the same is now true of the $I^\alpha_k$'s as well).
\end{proof}

\subsection{Proof of theorem \ref{eldecompthm}}
Let us denote the prescribed countable dense subset of $PS_h$ as in statement, by $\mathcal{C}$ enumerated as 
$\{ \alpha^n \; : \; n \in \mathbb{N} \}$, a countable dense subset of {\it distinct} points from $PS_h$. Next, suppose that a power series $g(z) = \sum c_J z^J $ with domain of convergence $D$ has been given. Let $c(\alpha)$ be the real valued function on $PS_N$ defined by the coefficients in $g$ as in the statement of the theorem; so,
\[
-c(\alpha) = \sup \big\{\limsup_{n \to \infty} \big(\frac{\log \vert c_{J^n} \vert}{\vert J^n \vert} \big) \; :
 \; \{J^n\} \in S_{\alpha} \big\}.
\]
Applying lemma \ref{cRealiztn} to each $\alpha^n$, we get a corresponding sequence $\{I^n_k\} \subset S_{\alpha^n}$ such that for each fixed $n \in \mathbb{N}$, we have
\[
\frac{\log \vert c_{I^n_k} \vert}{\vert I^n_k \vert} \to -c(\alpha^n)
\]
as $k \to \infty$. Note that for each $n$, the sequence $\pi(I^n_k) = I^n_k/\vert I^n_k \vert$ converges to $\alpha^n$, as $k \to \infty$.
Using these sequences, we split-up $PS\mathbb{Q}^N$ into 
sequences $\{R^n_k\}$ with $R^n_k \to \alpha^n$ as $k \to \infty$, for each $n \in \mathbb{N}$. Indeed, begin with $\{R^1_k\}$ given by the projection $R^1_k = \pi(I^1_k) \subset PS\mathbb{Q}^N$ which as noted above converges as $k \to \infty$ to the point $\alpha^1 \in \mathcal{C} \subset PS_N$. Next, as $\alpha^2 \neq \alpha^1$ and $\pi(I^2_k) \to \alpha^2$, we note that the sequence $\pi(I^2_k)$ can have atmost finitely many members in common with $\{\pi(I^1_k)\}$; discard away these common terms from $\pi(I^2_k)$ and reindex the sequence $I^2_k$ such that the resulting sequence $R^2_k:= \pi(I^2_k)$ has no terms in common with $\pi(I^1_k)$. Note that this simultaneously also ensures that $\{I^2_k\}$ is also disjoint from the sequence $\{I^1_k\}$. Further next, as before recalling that $\alpha^3$ is distinct from $\alpha^2,\alpha^1$ and 
$\pi(I^3_k) \to \alpha^3$, we note that the sequence $\pi(I^3_k)$ can have atmost finitely many members in common with $\{\pi(I^2_k)\}, 
\{\pi(I^1_k)\}$. Discard these common members from $\pi(I^3_k)$ and reindex the sequence $I^3_k$ such that the resulting sequence $R^3_k:= \pi(I^3_k)$ has no terms in common with $\pi(I^3_k)$. Repeating this process, we have after $l$ steps, mutually disjoint sequences (of distinct members)
$\{R^n_k : n \in \mathbb{N}\} \subset PS\mathbb{Q}^N$ for each $n=1,2, \ldots,l$ such that $R^n_k \to \alpha^n$ as $k \to \infty$. As we repeat the process any number of times ($l$ is arbitrary), we have thus generated (out of the $I^n_k$ as above), a doubly-indexed sequence of rational points $\{R^n_k\; : \; n,k \in \mathbb{N}\} \subset PS\mathbb{Q}^N $ which are distinct and such that for all $n \in \mathbb{N}$, $R^n_k \to \alpha^n$ as $k \to \infty$. Now, it may have happened that this doubly indexed sequence does not exhaust $PS \mathbb{Q}^N$; in such a case, we gather together all the left out rational points and enumerate them in a sequence $r_n$ of distinct members. We then prefix $r_n$ to the sequence $\{R^n_k : k \in \mathbb{N}\}$. Calling the resulting sequences still as $\{R^n_k\}$, we have thus achieved a decomposition of $PS \mathbb{Q}^N$ into mutually disjoint sequences of distinct members $\{R^n_k\}$ (whose union is $PS \mathbb{Q}^N$) with as before, $R^n_k \to \alpha^n$ as $k \to \infty$.
Now, let $\mathcal{R}^n= \{ R^n_k \; : \; k \in \mathbb{N}\}$. Then for each 
$n \in \mathbb{N}$, enumerate the countable set $\pi^{-1}( \mathcal{R}^n)$
as a sequence $\{J^n_k\}$ of distinct elements; note that the projection of 
$\{J^n_k\}$ on $PS_N$ 
is the convergent sequence $\{R^n_k\}$, converging to $\alpha^n$ as $k \to \infty$ i.e., $\{J^n_k\} \in S_{\alpha^n}$.  This implies that 
the series
\[
g_n(z) := \sum\limits_{J \in \pi^{-1}(\mathcal{R}^n)} c_J z^J 
\;\;=\sum\limits_{k \in \mathbb{N}} c_{J^n_k} z^{J^n_k}
\]
(where the coefficients $c_J$ of the monomial $z^J$ are exactly those in the given power series $g$) which is a sub-series of the given $g$, has as its domain of convergence by lemma \ref{half-space-Lem}, the elementary domain
$D_n = \lambda^{-1}(G_n)$ where $G_n$ is the half-space 
\[
G_n = \{ s \in \mathbb{R}^N \; : \; \langle \alpha^n, s \rangle + 
\limsup_{k \to \infty} \log \vert c_{J^n_k} \vert^{1/\vert J^n_k \vert } <0 \}.
\]
Now, the definition of $c(\alpha)$ and the fact that $\{J^n_k\} \in S_{\alpha^n}$ gives 
$\log \vert c_{J^n_k} \vert^{1/\vert J^n_k \vert } \leq -c(\alpha^n)$.
However, recalling the construction of the $J^n_k$'s and observing that for each fixed 
$n \in \mathbb{N}$, the sequence $\{J^n_k \; : \; k \in \mathbb{N}\}$ contains our original sequence $\{I^n_k\}$ as a subsequence, we actually deduce that 
\[
\limsup_{k \to \infty} \log \vert c_{J^n_k} \vert^{1/\vert J^n_k \vert } = -c(\alpha^n).
\]
Recall that by theorem \ref{charbysuppfn}, the convex closure of $c(\alpha)$ is the support function $h_G$ of $G$. As the convex closure of a function cannot exceed it, we have that $G_n$ contains the supporting half-space 
for $G$ with normal vector $\alpha^n$, namely 
\[
H_n=\{ s \in \mathbb{R}^N \; : \; \langle \alpha^n,s \rangle - h(\alpha^n)<0\}.
\]
While this finishes the proof for the most part, note that we haven't yet checked something more basic namely that the sum of the $g_n$'s actually equals $g$ i.e., we have achieved the desired decomposition of $g$. To this end, first note that as the $R^n_k$'s exhaust $PS\mathbb{Q}^N$, we likewise have
\[
\bigcup_{n \in \mathbb{N}} \pi^{-1}(\mathcal{R}^n) = \mathbb{N}_0^N \setminus \{0\}.
\]
Setting $g_0(z)$ to be the constant function given by the constant term in $g$, we note this means that, every one of the monomials $c_Jz^J$ as $J$ varies through $\mathbb{N}_0^N$ 
(whether $c_J=0$ or not), gets accounted for, in the sense that it occurs as a summand in 
one of the $g_n$'s; indeed in exactly one of them. This inturn means that the sum of the 
$g_n$'s equals the given $g$, provided only that we absorb the constant $g_0$ into any one of the remaining series $g_n$ for $n \geq 1$; for instance, if we absorb $g_0$ into $g_1$ we have:
\begin{equation} \label{eldecomp}
g(z) =  \big(g_0+g_1(z)\big) + g_2(z) + g_3(z) + \ldots
\end{equation}
Indeed, observe that the partial sums of this series $g_0 + g_1 + \ldots + g_n$ are among the partial sums of the given power series $g$ and hence, converges (absolutely and uniformly on compacts) in $D$ to $g$, thus showing that (\ref{eldecomp}) is the sought-for decomposition expressing the given power series as a sum of 
`elementary' power series, as asserted by the theorem, finishing the proof. \qed 

\noindent While the converse of this theorem as stated in the introduction is trivial, one can use theorem \ref{charbysuppfn} and the fact that support functions are continuous on their relative interior to formulate a better converse which assumes less, as follows. We omit the proof as it 
is fully straightforward.

\begin{prop} \label{partialconverse}
Let $\{g_n\}$ be a sequence of power series (not necessarily elementary) which are mutually monomial-wise completely different. Let 
$g$ denote their formal sum, written as a standard power series $\sum c_J z^J$ -- note that there are no `like-terms' to be combined as the 
$g_n$'s are mutually monomial-wise completely different and so, every non-zero coefficient $c_J$ and its associated monomial $c_J z^J$ occurs 
exactly in one of the $g_n$'s. Suppose that the coefficients of the $g_n$'s come together to satisfy the following: the convex closure of the 
function defined by
\[
c(\alpha) =  - \sup \big\{\limsup_{n \to \infty} \big(\frac{\log \vert c_{J^n} \vert}{\vert J^n \vert} \big) \; : \; \{J^n\} \in S_{\alpha} \big\},
\]
agrees with the support function $h=h_G$ of a convex domain $G \subset \mathbb{R}^N$ on a countable dense subset of $PS_h$, 
then $D=\lambda^{-1}(G)$ is  the domain of convergence of $g$. If $c$ itself is convex, then the domains of convergence of the $g_n$'s are pull-backs (under the standard logarithmic map $\lambda$) of half-spaces in $\mathbb{R}^N$ and the intersection of these elementary Reinhardt domains is $D$.
\end{prop}

\begin{rem}\label{remonc-condn}
The above proposition is an answer towards the question of when is the domain of convergence of an infinite sum of power series, the intersection of the domains of convergence of the individual summands? For this, the requirement that the coefficients of the $g_n$'s come together to satisfy the condition on the associated function $c$, as stated in the theorem is necessary, even if the other condition that they be mutually monomial-wise completely different is satisfied. For otherwise, a series as simple as the geometric series $\sum z^J$ becomes a counter-example; indeed, the domains of convergence of each of the summands which we just take to be the
 monomials $z^J$ is $\mathbb{C}^N$, whereas the domain of convergence of their sum is 
 just the unit polydisc!
\end{rem}

\noindent By summing up suitable elementary power series based on the foregoing ideas, we now construct a power series with its domain of convergence equal to a prescribed logarithmically convex complete Reinhardt domain to prove theorem \ref{mainthm} in the following subsection.

\subsection{Proof of theorem \ref{mainthm}}
\noindent The proof of theorem \ref{mainthm} requires the following lemma.
\begin{lem} \label{ratrectif}
Let $\mathcal{C}$ be any countable dense subset of $PS_N$ enumerated 
as $\mathcal{C} = \{ \alpha^n \; : \; n \in \mathbb{N} \}$. Then, there exists a  doubly-indexed sequence $\{J^{nk} \; : \; n, k \in \mathbb{N} \}$ of
distinct elements of $\mathbb{N}_0^N \setminus \{0\}$ such that:
for each $n \in \mathbb{N}$, 
$R^n_k := \pi(J^{nk})$ converges to $\alpha^n$ as $ k \to \infty$.
\end{lem}
\noindent We emphasize that $\{ J^{nl} : n , l \in \mathbb{N} \}$ is a (countable) set of
distinct points i.e., $J^{mk} \neq J^{nl}$ as soon as $m \neq n$ or $k \neq l$. Consequently,  for each $l \in \mathbb{N}$ and, for any $m =1,2, \ldots,l$,
\[
\{ J^{mk}\} \subset  \mathbb{N}_0^N \setminus \bigcup\limits_{i=1}^{m-1}\{J^{ik}:k \in \mathbb{N}\}.
\]
Note however, that the $R^n_k$'s need not be distinct; indeed, in the simple case of
$\mathcal{C} = \{ \alpha^n \; : \; n \in \mathbb{N}\}$ being precisely the rational points i.e., $\mathcal{C} = PS \mathbb{Q}^N$, the lemma above holds perfectly well
with the $J^{nk}$'s taken to be all of $\mathbb{N}_0^N \setminus \{0\}$, in which case for each $n$,
 $\{ R^n_k \; :\; k \in \mathbb{N}$ is a constant sequence. In all cases, $R^n_k = J^{nk}/ \vert J^{nk} \vert$ depends on 
$\mathcal{C}$. \\
\begin{proof}
Pick any sequence $\{J^{1k} \}$ of distinct elements of $\mathcal{N}:=\mathbb{N}_0^N \setminus \{0\}$ with 
$\pi(J^{1k}) \to \alpha^1$ as $k \to \infty$, where $\pi$ denotes as always in this article, the `radial' projection onto $PS_N$ given by $\pi(z)=z/\vert z \vert_{l^1}$. Explicitly:
\[
\Big( \frac{J^{1k}_1}{\vert J^{1k} \vert}, \ldots, \frac{J^{1k}_N}{\vert J^{1k} \vert} \Big) \to (\alpha^1_1, \ldots, \alpha^1_N),
\]
as $k \to \infty$ -- it is trivial to see that such a sequence exists. Next, pick a sequence $\{ J^{2k}\}$, this time in $\mathcal{N} \setminus \{J^{1k} \}$ of distinct points, such that
\[
\Big( \frac{J^{2k}_1}{\vert J^{2k} \vert}, \ldots, \frac{J^{2k}_N}{\vert J^{2k} \vert} \Big) \to (\alpha^2_1, \ldots, \alpha^2_N).
\]
Such a sequence exits, as $\pi\big( \mathcal{N}\setminus \{J^{1k} \} \big)$ is dense in $PS_N \setminus \{\alpha^1\}$. After $l$ steps, we would have sequences $\{J^{lk} \; : \; k \in \mathbb{N}^N\}$ such that for any $m \leq l$, we have 
\[
\{ J^{mk}\} \subset  \mathbb{N}_0^N \setminus \bigcup\limits_{i=1}^{m-1}\{J^{ik}:k \in \mathbb{N}\}
\]
and $J^{mk}/ \vert J^{mk} \vert \to \alpha^m$ as $k \to \infty$. 
Letting $l \to \infty$, this iterative procedure generates the desired doubly-indexed sequence $J^{nk}$ (as $n,k$ both vary in $\mathbb{N}$) in which of no two members are equal and completes the proof.
\end{proof}
\begin{proof}[Proof of theorem \ref{mainthm}]
The half-spaces $H_n$'s which are the logarithmic images of the domains of convergence of the $f_n$'s as in the statement of theorem \ref{mainthm} can be written down by lemma \ref{half-space-Lem}, as 
\begin{align*}
H_n &= \{ s \in \mathbb{R}^N \; : \; \langle \alpha^n, s \rangle + \limsup_{k \to \infty} \log \vert c_{nk} \vert^{1/\vert J^{nk} \vert} \; <0\} \\
&= \{ s \in \mathbb{R}^N \; ; \; \langle \alpha^n, s \rangle + \limsup_{k \to \infty} \log (e^{-h(\alpha^n)}) \; <0 \}\\
&= \{ s \in \mathbb{R}^N \; ; \; \langle \alpha^n, s \rangle - \liminf_{k \to \infty} h(\alpha^n)  \; <0 \}
\end{align*} 
But then as $h(\alpha^n)$ is independent of $k$, the liminf operation in the above may well be dropped to get that 
\[
H_n = \{ s \in \mathbb{R}^N \; ; \; \langle \alpha^n, s \rangle -  h(\alpha^n)  \; <0 \}
\]
which is actually indeed a supporting half-space for $G$, the logarithmic image of the given domain $D$ as in the statement of theorem \ref{mainthm}. Now, observe that if the power series $f$ converges absolutely at $z$, then so does each of the $f_n$'s by virtue of them being mutually monomial-wise completely different; as the domain of convergence of a power series is the domain of absolute convergence as well, we conclude that $G_f \subset H_n$ for all $n$, where $G_f$ denotes the domain of convergence of our power series $f$.
Recalling that by theorem \ref{cvxgeomfound}, these countably many half-spaces $H_n$ suffice to render $G$ by their intersection, 
we therefore have
\begin{equation} \label{FundIncl}
G_f \subset \bigcap_{n=1}^{\infty} H_n = G.
\end{equation}
To obtain the reverse inclusion, we shall compare the support functions of $G_f$ and $G$. To this end, recall that by theorem \ref{charbysuppfn}, the support function $h_f$ of $G_f$ equals the convex closure of a function $c=c_f$ as defined in that theorem and in particular along the sequence 
$\{\alpha^n\}$, we have
\begin{equation}\label{cfeqn}
c_f(\alpha^n) = 
- \sup \big\{\limsup_{m \to \infty} \big(\frac{\log \vert c_{J^m} \vert}{\vert J^m \vert} \big) \; : \; \{J^m\} \in S_{\alpha^n} \big\}.
\end{equation}
For any fixed $\alpha^n$, note that by definition of the coefficients in our power series $f$, we will have the following subsequential limit occuring in the evaluation of the limsup in (\ref{cfeqn}):
\begin{align*}
\limsup_{k \to \infty} \frac{\log \vert c_{J^{nk} } \vert}{\vert J^{nk} \vert} = \limsup_{k \to \infty} 
\frac{\log \vert e^{-\vert J^{nk} \vert h(\alpha^n)} \vert}{\vert J^{nk} \vert} = -h_G(\alpha^n)
\end{align*}
where we have denoted the support function of $G$ by $h_G$ (rather than just by $h$ as above, for some clarity here as there are multiple support functions lurking). This when used in (\ref{cfeqn}) gives $-c_f(\alpha^n) \geq -h_G(\alpha^n)$ i.e., we have $c_f \leq h_G$ on the countable dense subset $\{\alpha^n\}$ of $PS_h$ (where again $h=h_G$). As $h_f$ is a minorant of $c_f$, we have the same inequality with $c_f$ replaced by $h_f$. But then by the density of $\{\alpha^n\}$ in $PS_h$ and the everywhere lower semicontinuity of support functions suffices to ensure persistence of the inequality throughout $PS_h$:
\[
h_f(\alpha) = \liminf_{k \to \infty} h_f(\alpha^{n_k}) \leq \liminf_{k \to \infty} h(\alpha^{n_k}) = h_G(\alpha).
\]
This means in particular that $PS_{h_f} \supset PS_{h_G}$ and more importantly, that 
\begin{align*}
G_f &= \bigcap_{\alpha \in PS_{h_f}} \{s \in \mathbb{R}^N \; : \; \langle \alpha, s \rangle - h_f(\alpha) <0\} \\
&\subset \bigcap_{\alpha \in PS_{h_G}} \{s \in \mathbb{R}^N \; : \; \langle \alpha, s \rangle - h_G(\alpha) <0\}
\end{align*}
yielding that $G_f \subset G$. This when combined with (\ref{FundIncl}) renders the desired equalities: $\tilde{G} = G = G_f$, finishing the proof that the domain of convergence of $f$ is precisely the given domain $D$.
\end{proof}

\noindent As mentioned following the statement of theorem \ref{mainthm} in the introduction, we may now deduce as a corollary here, the 
well-known fact that logarithmically convex complete Reinhardt domains $D$ are domains of holomorphy. Since such domains are completely determined by their absolute images $\vert D \vert$, it suffices to show that for every point $p$ in the absolute boundary $\partial \vert D\vert$ (which are the boundary points of $D$ with non-negative coordinates as well), there exists a holomorphic function on $D$ which is completely singular at $p$;
indeed, if $p=(r_1e^{i \theta_1}, \ldots,r_Ne^{i \theta_N})$ is an arbitrary boundary point of $D$ and $f$ is a holomorphic function on $D$ which is completely singular at $\vert p \vert=(r_1, \ldots,r_N)$, then $f_p(z):=f(e^{-i\theta_1}z_1, \ldots, e^{-i\theta_N}z_N)$ is completely singular at
the given point $p$. Now, it only remains to note that the existence of holomorphic functions singular at the absolute boundary points is manifestly furnished by the power series of theorem \ref{mainthm} whose coefficients are all positive in view of lemma \ref{multiPringsheim}, the multidimensional version of Pringsheim's theorem.\\

\noindent While we are at this point of furnishing simple examples of power series for prescribed domains $D$ of convergence, let us settle a natural question that may be asked here namely, if it is possible to  give an example of power series whose domain of convergence is precisely a given domain $D$ but whose singular set in $\partial D$ is as small as possible; the other extreme namely when the singular set equals $\partial D$
is also no less interesting, infact only more so, but has been already dealt with and recorded well in \cite{Sic}. In the case $N=1$, this is too simple, as shown by the geometric series which is singular at exactly one point on the boundary of its disc of convergence and this settles the question of minimality of the singular set which has always got to be non-empty. As soon as $N>1$, it is impossible for a power series to be singular at just one point on the boundary of its domain of convergence or for it to have only a countable singular set, as implied by the following.

\begin{prop} \label{singsetarbpow}
Let $f$ be any power series whose domain of convergence $D$ is neither empty nor the whole space $\mathbb{C}^N$.  Then,
$\tau(S_f) = \tau(\partial D)$, where $\tau: \mathbb{C}^N \to \mathbb{R}_+^N$ is the absolute map 
$\tau(z) = (\vert z_1\vert, \ldots, \vert z_N \vert)$.
\end{prop}

\begin{proof}
To begin with observe that the absolute map $\tau$ when restricted to $D$ is a proper mapping of $D$ onto $\tau(D)$, as well as between their closures. Thus in particular, this restricted map is a closed map and so, $\tau(S_f)$ is a closed subset of $\tau(\partial D)$
Therefore to prove the proposition, it suffices to show that every point $p$ on the strictly positive part of the boundary $\partial D$ lies in $\tau(S_f)$, since such points form a dense subset of $\tau(\partial D)$; here and in the sequel, by the strictly positive part of the 
boundary $\partial D$, we mean the subset of $\partial D$ consisting of strictly positive points i.e., those points with the property that all of their coordinates are strictly positive reals.
Now, we shall argue by contradiction and to this end suppose $p$ is a strictly positive point of $\partial D$ which does not lie in $\tau(S_f)$. This means that the
fiber over $p$ under the mapping $\tau$ namely, the $N$-dimensional torus $\tau^{-1}(p)$, consists only of regular points. Considering the union of the neighbourhoods of each such regular point where $f$ is guaranteed to have a direct analytic continuation, we obtain a Reinhardt domain $V$ (indeed a tubular neighbourhood) containing the compact set $\tau^{-1}(p)$ to which $f$ has a holomorphic extension which we continue to denote by $f$. As
 $\tau$ is an open map, $U:=\tau(V)$ is a neighbourhood in the absolute space of the point $p$. Note that $U$ itself is contained in $V$, indeed $U= V \cap \mathbb{R}_+^N$ where $f$ is well-defined. Let $\delta_0>0$ be such that the $l^\infty$-ball (which is ofcourse a polydisc) of radius 
 $\delta_0$ centered at the point $p$ lies within $U$.  We claim that there exists $\delta \in (0, \delta_0)$ and $j\in \{1,2,\ldots,N\}$ such that if we increase  the $j$-th coordinate of $p$ by $\delta$ while keeping the rest fixed, then the resulting point $p_\delta^j$ will be point which lies in 
 $U \setminus \overline{D}$. The truth of this claim can be verified by an argument by contradiction. Indeed if the claim were false, then it means that for every $\delta \in (0, \delta_0)$ and for {\it every} $j\in \{1,2,\ldots,N\}$, the points $p_\delta^j$ obtained by increasing the $j$-th coordinate of $p$
by $\delta$, all lie in $\overline{D}$. This inturn means that the logarithmically convex hull $H^L_p$ of the points $p_\delta^1, \ldots, p_\delta^N$ must itself be (fully) contained in $\overline{D}$ as $D$ and thereby $\overline{D}$ are logarithmically convex. But then $p$ is an interior point of $H^L_p$ implying that $p$ must be an interior point of $D$, contradicting that $p$ is a boundary point of $D$. This proves the claim that there exists $j\in \{1,2,\ldots,N\}$, such that $p_\delta^j$ lies in $U \setminus \overline{D}$. We assume for convenience that 
$j=N$ (arguments for other possibilities for $j$ are essentially the same, needless to say) and write $p_\delta^N$ simply by $q$ henceforth. By shrinking $\delta$ if necessary, we can choose an (open) `$\delta$-box' $B$ with the points $q$ and $p$ both on the boundary of this box and with 
$\overline{B} \subset U$, where in referring to $B$ as a `$\delta$-box', we mean that $B$ is of the form
\[
B=\{ z \in \mathbb{C}^N \; : \; p_j - \delta < \vert z_j \vert < p_j \; \text{ for all} \; j=1,\ldots,N-1\; \textrm{ and }
p_N < \vert z_N \vert < p_N + \delta \}.
\]
So this `box' is actually a product of annuli in $\mathbb{C}$ but we shall just refer to it as a box for convenience. Note the difference in the qualifying condition on the $N$-th coordinate of a point to qualify to be in $B$ as compared to the defining conditions on the other coordinates. Note also that $q \in \partial B$ which means in particular that the box $B \subset U$ is not contained in $\overline{D}$. Recalling that the entire
closed polydisc $\overline{P}_p$ spanned by the point $p$ is contained in $\overline{D}$ as $\overline{D}$ is a Reinhardt set, we then observe that
$\overline{B} \cup \overline{P}_p \subset \overline{D} \cup U$ and more importantly, that the following `elongated open box' $\tilde{B}$ given by
\[
\tilde{B} = \{ z \in \mathbb{C}^N \; : \; p_j - \delta < \vert z_j \vert < p_j \; \text{ for all} \; j=1,\ldots,N-1\; \textrm{ and }
0 < \vert z_N \vert < p_N + \delta \}.
\]
 is contained
within $D \cup U$. Adjoining to this, a polydisc $\tilde{P}$ of the form
\[
\tilde{P}= \{ z \in \mathbb{C}^N \; : \; 0  \leq \vert z_j \vert < p_j \; \text{ for all} \; j=1,\ldots,N-1\; \textrm{ and } 
0 \leq \vert z_N \vert < \delta \}
\]
gives the Hartogs figure $H=\tilde{P} \cup \tilde{B}$. The foregoing containment observations guarantee that this Hartogs gigure $H$ is contained within $D \cup U$, where our originally given power series has been extended holomorphically and which we are continuing to denote by $f$. By the standard technique employing the Cauchy integral formula and the identity principle, we get a holomorphic extension $f_{\vert_H}$ to the polydisc $P_q$ spanned by $q$ (which agrees with the given $f$ on $P_q \cap D$). But then as $P_q$ is a polydisc centered at the origin,  $D \cup P_q$ is a complete Reinhardt domain and as $f$ has now been holomorphically extended to $D \cup P_q$, we have by proposition \ref{singlepowser} that the originally given power series representing $f$ must actually converge on $D \cup P_q$ in particular, contradicting that $D$ was the domain of convergence. Our proof by contradiction is now complete and as noted at the outset this means that $\tau(S_f)=\tau(\partial D)$.
\end{proof}

\subsection{Proof of theorem \ref{structurethm}}
As before, denote the prescribed countable dense subset of $PS_h$ as in statement, by $\mathcal{C}$ enumerated as 
$\{ \alpha^n \; : \; n \in \mathbb{N} \}$, a countable dense subset of {\it distinct} points from $PS_h$. Next, suppose that a power series $g(z) = \sum c_J z^J $ with domain of convergence $D$ has been given. Let $c(\alpha)$ be the real valued function on $PS_N$ defined by the coefficients in $g$ as in the statement of the theorem; so,
\[
-c(\alpha) = \sup \big\{\limsup_{n \to \infty} \big(\frac{\log \vert c_{J^n} \vert}{\vert J^n \vert} \big) \; :
 \; \{J^n\} \in S_{\alpha} \big\}.
\]
Applying lemma \ref{cRealiztn} to each $\alpha^n$, we get a corresponding sequence $\{I^n_k\} \in S_{\alpha^n}$ of distinct elements of 
$\mathbb{N}_0^N \setminus \{0\}$ such that for each fixed $n \in \mathbb{N}$, we have
\[
\frac{\log \vert c_{I^n_k} \vert}{\vert I^n_k \vert} \to -c(\alpha^n)
\]
as $k \to \infty$. Note that for each $n$, the sequence $\pi(I^n_k) = I^n_k/\vert I^n_k \vert$ converges to $\alpha^n$, as $k \to \infty$. 
The distinctness of $\alpha^n$'s implies that the sequences $\{I^n_k : k \in \mathbb{N}\}$ are mutually disjoint so that altogether, the doubly indexed sequence $I^n_k$ consists only of distinct members. To redress the possibility that the $I^n_k$'s may fail to exhaust $\mathbb{N}_0^N \setminus \{0\}$, we enumerate the left out elements as a sequence $\{P_n\}$ (of distinct elements) and prefix the $n$-th member of this sequence i.e., $P_n$ to the sequence $I^n_k$ and call the resulting sequence still as $I^n_k$. For all $n \in \mathbb{N}$, set
\[
g_n(z)= \sum\limits_{k=1}^{\infty} c_{I^n_k} z^{I^n_k},
\]
where $c_{I^n_k}$ is the coefficient of the monomial $z^{I^n_k}$ in the given power series $g(z)$; so $g_n$ is a sub-series of $g$. Note 
that $g_n$ is an elementary power series, the logarithmic image of whose domain of convergence is the half-space given by
\[
G_n= \{ s \in \mathbb{R}^N \; : \; \langle \alpha^n, s \rangle  + \limsup_{k \to \infty} \log \vert c_{I^n_k} \vert^{1/\vert I^n_k \vert } <0 \}.
\]
While we know that 
\begin{equation}\label{upbdforhalfspcoeff}
\limsup_{k \to \infty} \log \vert c_{I^n_k} \vert^{1/\vert I^n_k \vert } \leq h_G(\alpha^n)
\end{equation}
we cannot claim equality here. On the other hand, equality would hold if we considered a different power series namely, 
\[
f_n(z) = \sum_{k \in \mathbb{N}} c_{nk}z^{I^n_k} 
\]
with $c_{nk}= e^{-\vert I^n_k \vert h(\alpha^n)}$ for all $n,k \in \mathbb{N}$. That is to say, note that the domain of convergence of $f_n$ has 
its logarithmic image equal to the half-space given by 
\[
H_n=\{ s \in \mathbb{R}^N \; ; \; \langle \alpha^n, s \rangle -  h_G(\alpha^n)  \; <0 \}
\]
as deduced in the proof theorem \ref{mainthm}. To obtain the desired decomposition as asserted by the theorem being proved, let 
$\sigma_1=g_1 + f_1$, $\sigma_2=g_1+f_2/2-f_1$, $\sigma_3=g_3 + f_3/3 - f_2/2$ and in general, for all $n \in \mathbb{N}$, let
\begin{equation} \label{sigmadefn}
\sigma_{n+1} = \big(g_{n+1} + \frac{f_{n+1}}{n+1}\big) - \frac{f_{n}}{n}.
\end{equation}
Note that every finite sum of the $\sigma_n$'s is of the form
\[
S_n(z) = \sigma_1(z) + \sigma_2(z) + \ldots + \sigma_n(z) = g_1(z) + g_2(z) + \ldots + g_n(z) + \frac{1}{n}f_n(z).
\]
Observe then that $S_n(z) \to g(z)$ as $n \to \infty$, absolutely and uniformly on compacts subsets of $D$.
 
By (\ref{upbdforhalfspcoeff}), we have $G_n \supset H_n$. By lemma \ref{samepairlem}, the domain of convergence of the powers series given by the sum 
$\tilde{\sigma}_n=g_n +f_n/n$ is just $H_n$, the supporting half-space for $G$ with normal vector $\alpha^n$. Note that the
power series $\tilde{\sigma}_n$ is obtained by combining all the like terms in $g_n$ and $f_n/n$ and once this has been done,
by (\ref{sigmadefn}) the power
   series $\sigma_n$ is the sum of two power series $\tilde{\sigma}_n$ and $-f_{n-1}/(n-1)$ which are mutually monomial-wise completely different. Hence by lemma \ref{domcvgcesuminter}, the domain of convergence of $h_n$ is the intersection of these two series which is precisely 
\[
D_n = \lambda^{-1}(H_n) \cap \lambda^{-1}(H_{n-1})= \lambda^{-1}(H_n \cap H_{n-1}),
\]
the pull-back of the wedge $W_n =H_n \cap H_{n-1}$ formed by the intersection of the pair of half-spaces $H_n, H_{n-1}$ of different normal vectors from the prescribed set $C$ as in statement of the theorem. As the intersection of $H_n$'s (thereby that of $W_n$'s as well) is precisely $G$, we conclude that the intersection of the simple domains $D_n$'s is precisely the given domain $D$ and $\sigma_1 + \sigma_2 + \sigma_3 +\ldots$, gives the desired decomposition of the given power series $g$ into simple power series, finishing the proof.\qed


\begin{thebibliography}{BFKKMP}

\bibitem{AM} L. A. Aĭzenberg, B. S. Mitjagin ,
\textit{Spaces of functions analytic in multi-circular domains}. (Russian)
Sibirsk. Mat. Ž., 1 (1960), pp 153 -- 170.

\bibitem{B} G. P. Balakumar,
\textit{Power Series in Several Complex Variables}
https://arxiv.org/abs/1601.00274 


\bibitem{Bo} Harold Boas,
\textit{Math 650: Several Complex Variables, Fall-2013.}, http://www.math.tamu.edu/~boas/courses/650-2013c/

\bibitem{FG} Klaus Fritzsche, Hans Grauert:
\textit{From Holomorphic Functions to Complex Manifolds}, Volume 213 of Graduate Texts in Mathematics, Springer, 2002.

\bibitem{G} Peter Gruber,
\textit{Convex and Discrete Geometry}, Volume 336 of Grundlehren der mathematischen Wissenschaften, Springer, 2007.

\bibitem{Hartogs} F. Hartogs,
\textit{Zur Theorie der analytischen Funktion mehrerer unabhängiger Veränderlichen, insbesondere über die Darstellung derselben durch Reihen, welche nach Potenzen einer Veränderlichen fortschreiten}. (German), Math. Ann. 62, 1--88 (1906).

\bibitem{Rocka} Rockafellar R. T.,
\textit{Convex Analysis}, Princeton University Press, 1970.

\bibitem{JarPflu} Marek Jarnicki, Peter Pflug,
\textit{First steps in several complex variables: Reinhardt domains},
EMS Textbooks in Mathematics. European Mathematical Society (EMS), Zürich, 2008. viii+359 pp. 

\bibitem{H}  Lars H\"ormander,
\textit{Notions of convexity}, Modern Birkh\"auser Classics. Birkhäuser Boston, Inc., Boston, MA, 2007.


\bibitem{R} R. Michael Range, 
\textit{Holomorphic functions and integral representations in several complex variables},
Graduate Texts in Mathematics, 108. Springer-Verlag, New York, 1986. xx+386 pp. 

\bibitem{S}  B. V. Shabat,
\textit{Introduction to complex analysis. Part II. Functions of several variables.} Translated from the third (1985) Russian edition by J. S. Joel. Translations of Mathematical Monographs, 110. American Mathematical Society, Providence, RI, 1992. x+371 pp.

\bibitem{SS} Alexander D. Scott, Alan D. Sokal, 
\textit{The repulsive lattice gas, the independent-set polynomial, and the Lovász local lemma}, J. Stat. Phys. 118 (2005), no. 5-6, 1151--1261.

\bibitem{Sic}  Józef Siciak,
\textit{Some gap power series in multidimensional setting}, Ann. Univ. Mariae Curie-Skłodowska Sect. A 65 (2011), no. 2, 179--190. 

 
\bibitem{W} Hung-Hsi Wu,
\textit{The spherical images of convex hypersurfaces}, 
J. Differ. Geom. 9, 279--290 (1974).






\end{thebibliography}
\end{document}